\documentclass{amsart}
\usepackage{amsfonts}
\usepackage{mathrsfs}
\usepackage{amsmath,amsthm,amscd,mathrsfs,amssymb,graphicx,enumerate}
\usepackage{xypic}

\newtheorem{thm}{Theorem}[section]
\newtheorem{cor}[thm]{Corollary}
\newtheorem{lem}[thm]{Lemma}
\newtheorem{prop}[thm]{Proposition}
\theoremstyle{definition}
\newtheorem{defn}[thm]{Definition}

\newtheorem{rmk}[thm]{Remark}

 \DeclareMathOperator{\rk}{rk}
 \DeclareMathOperator{\Sym}{Sym}
\DeclareMathOperator{\Pic}{Pic} 

\newcommand{\C}{\ensuremath\mathbb{C}}
\newcommand{\R}{\ensuremath\mathbb{R}}
\newcommand{\Z}{\ensuremath\mathbb{Z}}
\newcommand{\Q}{\ensuremath\mathbb{Q}}

\newcommand{\fa}{\ensuremath\mathfrak{a}}
\newcommand{\fb}{\ensuremath\mathfrak{b}}

\newcommand{\PP}{\ensuremath\mathbb{P}}
\newcommand{\calO}{\ensuremath\mathcal{O}}

\newcommand{\Set}[2]{\left\{#1:#2\right\}}

\begin{document}
\title{Prym-Tjurin constructions on cubic hypersurfaces}
\author{Mingmin Shen}
\address{DPMMS, University of Cambridge, Wilberforce Road, Cambridge CB3 0WB, UK}
\email{M.Shen@dpmms.cam.ac.uk}


\date{}

\begin{abstract}
In this paper, we give a Prym-Tjurin construction for the cohomology
and the Chow groups of a cubic hypersurface. On the space of lines
meeting a given rational curve, there is the incidence
correspondence. This correspondence induces an action on the
primitive cohomology and the primitive Chow groups. We first show
that this action satisfies a quadratic equation. Then the
Abel-Jacobi mapping induces an isomorphism between the primitive
cohomology of the cubic hypersurface and the Prym-Tjurin part of the
above action. This also holds for Chow groups with rational
coefficients. All the constructions are based on a natural relation
among topological (resp. algebraic) cycles on $X$ modulo homological
(resp. rational) equivalence.
\end{abstract}

\maketitle

\section{Introduction}
Algebraic cycles on a cubic hypersurface have been serving as an
interesting but nontrivial example in the study of Chow groups. In
\cite{shen}, we obtained natural relations among 1-cycles and gave
some applications of those relations. This paper is a continuation
of \cite{shen} and generalization of the results to higher
dimensional cycles on cubic hypersurfaces. Throughout the paper, we
will work over the field $\C$ of complex numbers. In this article,
all homology and cohomology groups are modulo torsion.

Let $X\subset\PP^{n+1}$ be a smooth cubic hypersurface of dimension
$n\geq 3$. We first establish Theorem \ref{fundamental relation}
which gives a natural relation among cycles on $X$. To get an idea
of what such a relation is, we first fix a smooth curve $C\subset X$
of degree $e$. Let $T\subset X$ be a closed subvariety of dimension
$r$. Let $T'\subset X$ be the subvariety swept out by lines on $X$
that meet both $C$ and $T$. When $C$ and $T$ are in general
position, then $T'$ has expected dimension $r$. If $r\geq 2$, then
the relation we get is
$$
2e\,T+T'=a\,h^{n-r},\quad\text{in }\mathrm{CH}_r(X)
$$
for some integer $a>0$, where $h$ is the class of hyperplane section
of $X$. If $r=1$, the relation we get is
$$
2e\,T+T'+2\deg(T)C=b\,h^{n-1},\quad\text{in }\mathrm{CH}_1(X)
$$
for some integer $b>0$. The second case was first proved in
\cite{shen}.

Let $F$ be the Fano scheme of lines on $X$. Let $p:P(X)\to F$ be the
total family of lines on $X$ and $q:P(X)\to X$ be the natural
morphism. Then we have the natural cylinder homomorphism
$\Psi=q_*p^*$ and its transpose, the Abel-Jacobi homomorphism,
$\Phi=p_*q^*$. These are defined on the cohomology groups
and the Chow groups. The following theorem is a more concise
expression of the above relations; see Corollary \ref{cor fundamental
relation}.
\begin{thm}
Let $\gamma\in\mathrm{H}_n(X,\Z)$, $\fa\in\mathrm{CH}_1(X)$ and
$\fb\in\mathrm{CH}_l(X)$. We use $h$ to denote the class of a
hyperplane in either the Chow group or the cohomology group of $X$.
Then the following are true.

(1)
$2\deg(\fa)\gamma+\Psi(\Phi(\gamma)\cdot\Phi([\fa]))=3\deg(\gamma)\deg(\fa)h^{\frac{n}{2}}$
in $\mathrm{H}_n(X,\Z)$, where the right hand side is $0$ if $n$ is
odd.

(2) If $n-2\geq l\geq2$, then
$2\deg(\fa)\fb+\Psi(\Phi(\fa)\cdot\Phi(\fb))=3\deg(\fa)\deg(\fb)h^{n-l}$
in $\mathrm{CH}_l(X)$.

(3) If $l=1$, then $2\deg(\fa)\fb +2\deg(\fb)\fa
+\Psi(\Phi(\fa)\cdot\Phi(\fb))=3\deg(\fa)\deg(\fb)h^{n-1}$ in
$\mathrm{CH}_1(X)$.
\end{thm}

\begin{rmk}
If $n\geq 5$, then $\mathrm{CH}_1(X)\cong \Z$; see \cite[Proposition 4.2]{paranjape}. In this case,
$\fa\in\mathrm{CH}_1(X)$ is the same as $\deg(\fa)[l]$ where $[l]$
is the class of a line. For fixed $n$, the groups $\mathrm{CH}_l(X)$
are expected to be trivial (isomorphic to $\Z$) when $l$ small or
big enough. Hence, the above relations are interesting when $l$ is
in the middle range.
\end{rmk}

With these relations, we will realize the Hodge structure and Chow
groups of $X$ as Prym-Tjurin constructions.

A Prym variety is constructed from a curve together with an
involution and they form a larger class of principally polarized
abelian varieties (p.p.a.v.) than the Jacobian of curves. Mumford
gives this a modern treatment in \cite{mumford}. He also uses Prym
varieties to characterize the intermediate Jacobian of a cubic
threefold, see the appendix of \cite{cg}. In \cite{Tjurin}, Tjurin
generalizes this idea by replacing the involution with a
correspondence that satisfies a quadratic equation. This gives what
we now call Prym-Tjurin varieties. This was further developed and
completed by S. Bloch and J.P. Murre, see \cite{bm}. Welters has
proved that all p.p.a.v.'s can be realized as Prym-Tjurin varieties,
see \cite{welters}. Roughly speaking, this means that every
principally polarized weight one Hodge structure can be realized via
a Prym-Tjurin construction on some curve. We naturally ask whether
similar constructions can be done for higher weight Hodge
structures. The work of Lewis in \cite{lewis} sheds some light on
this question. Izadi gives a Prym construction for the cohomology of
cubic hypersurfaces in \cite{izadi}.

\begin{defn}
Let $\Lambda$ be an abelian group and $\sigma:\Lambda\rightarrow
\Lambda$ an endomorphism. Assume that $\sigma$ satisfies the
quadratic equation $(\sigma-1)(\sigma+q-1)=0$. Then the
\textit{Prym-Tjurin part} of $\Lambda$ is defined as
$$
P(\Lambda,\sigma)=\mathrm{Im}(\sigma-1)
$$
\end{defn}

\begin{rmk}
In many cases when $\Lambda$ carries some extra structure such as
Hodge structure, the homomorphism $\sigma$ is usually compatible
with that extra structure and $P(\Lambda,\sigma)$ carries an induced
such structure. One very interesting case is when
$\Lambda=\mathrm{H}^*(Y,\Z)$ together with the natural Hodge
structure and $\sigma$ is an action that is induced by some
correspondence $\Gamma\in \mathrm{CH}^{\dim Y}(Y\times Y)$.
\end{rmk}
Let $X\subset\PP^{n+1}_\C$ be a smooth cubic hypersurface as above.
We start with a general rational curve $C\subset X$ of degree $e\geq
2$. Let $S_C$ be a natural resolution (in fact the normalization) of
the space of lines on $X$ that meet $C$. Let $F=F(X)$ be the Fano
scheme of lines on $X$, $p:P\to F$ the total family and $q:P\to X$
the natural morphism. Then $S_C=q^{-1}(C)$. In particular, we have a
natural morphism $q_0:S_C\to C$. Consider the family $p_C:P_C\to
S_C$ of lines parameterized by $S_C$ with the natural morphism
$q_C:P_C\to X$, we have the natural cylinder homomorphisms
$\Psi_C=(q_C)_*(p_C)^*:\mathrm{H}^{n-2}(S_C,\Z)\rightarrow
\mathrm{H}^n(X,\Z)$ and
$\Psi_C=(q_C)_*(p_C)^*:\mathrm{CH}_r(S_C)\rightarrow
\mathrm{CH}_{r+1}(X)$. Similarly we have its transpose
$\Phi_C=(p_C)_*(q_C)^*:\mathrm{H}^n(X,\Z)
\rightarrow\mathrm{H}^{n-2}(S_C,\Z)$ and
$\Phi_C=(p_C)_*(q_C)^*:\mathrm{CH}_{r+1}(X)\rightarrow\mathrm{CH}_r(S_C)$.
On $S_C$ we have the naturally defined incidence correspondence
which induces an action $\sigma$ on the cohomology groups
and the Chow groups. Note that there is a natural morphism from
$S_C$ to the Grassmannian $G=G(2,n+2)$. The Pl\"ucker embedding of
$G$ induces a ample class $g$ on $S_C$. On $S_C$, we have another
divisor class $g'=q_0^*(pt)$ for some $pt\in C$. Let
$\Q[g,g']\subset \mathrm{H}^*(S_C,\Q)$ be the subring generated by
$g$ and $g'$. We define the primitive cohomology (or cycle) classes
on $S_C$ to be those which are orthogonal to the classes in
$\Q[g,g']$. The primitive cohomology is denoted by
$\mathrm{H}^{*}(S_C,\Z)^\circ$ and the primitive Chow group is
denoted by $\mathrm{CH}_{*}(S_C)^\circ$. The subring $\Q[g,g']$ is
invariant under the action of $\sigma$ and hence $\sigma$ acts on
the primitive cohomology and also the primitive Chow groups. Our
main result is the following

\begin{thm}
Let $C\subset X$ be a general rational curve of degree $e\geq 2$ as
above. Let $\sigma$ be the action of the incidence correspondence on
either the primitive cohomology or the primitive Chow group. Then
the following are true.

(1) The action $\sigma$ satisfies the following quadratic relation
$$
(\sigma-1)(\sigma +2e-1)=0
$$

(2) The map $\Phi_C$ induces an isomorphism of Hodge structures
$$
\Phi_C:\mathrm{H}^{n}(X,\Z)_{\text{prim}} \rightarrow
P(\mathrm{H}^{n-2}(S_C,\Z)^\circ,\sigma)(-1)
$$
where the $(-1)$ on the right hand side means shifting of degree by
$(1,1)$. The intersection forms are related by the following
identity
$$
\Phi_C(\alpha)\cdot\Phi_C(\beta)=-2e\,\alpha\cdot\beta
$$

(3) The map $\Phi_C$ induces an isomorphism
$$
\Phi_C:\mathrm{A}_i(X)_\Q\rightarrow
P(\mathrm{CH}_{i-1}(S_C)^\circ_\Q,\sigma)
$$
\end{thm}

This is proved in section 4 (Theorem \ref{main theorem}). In
\cite{izadi}, Izadi has proved a variant of statement (2) for $C$
being a line. In \cite{shen}, we proved the above theorem for cubic
threefolds where (3) holds true for integral coefficients. Besides
the natural relations stated at the very beginning of this section,
another ingredient to carry out the Prym-Tjurin construction is the
surjectivity of $\Psi_C$ on primitive cohomologies.
\begin{thm}
The natural homomorphism
$$
\Psi_C:\mathrm{H}^{n-2}(S_C,\Z)^\circ\rightarrow
\mathrm{H}^n(X,\Z)_{prim}
$$
between the primitive cohomologies is surjective.
\end{thm}
This is proved in Section 5 using the Clemens-Letizia method (see
\cite{clemens} and \cite{letizia}).

\vspace{3mm}

\textit{Acknowledgement}. The author was supported by Simons
Foundation as a Simons Postdoctoral Fellow during the completion of
this work. Part of this work was completed when the author was
visiting Beijing International Center for Mathematical Research. He
thanks E.~Izadi for explaining her results of \cite{izadi}.

\vspace{4mm}
\noindent \textbf{Notation}:\\
\indent $X\subset \PP^{n+1}$, smooth cubic hypersurface of
dimension $n\geq 3$ with $h$ the hyperplane class;\\
\indent $G=\mathrm{G}(2,n+2)$, the Grassmannian of lines in
$\PP^{n+1}$;\\
\indent $F=F(X)\subset G$, Fano scheme of lines on $X$, smooth of
dimension $2n-4$; see \cite{cg} and also \cite{ak}.\\
\indent $l\subset X$, a line on $X$; $[l]\in F$ the corresponding point on $F$;\\
\indent $P=P(X)$, the universal family of lines on $X$, namely we
have the following diagram
$$
\xymatrix{
 P(X)\ar[r]^q\ar[d]_p &X\\
 F &
}
$$
\indent $I\subset F\times F$, the incidence correspondence, i.e.
the closure of $\{([l_1],[l_2]): l_1\neq l_2,\, l_1 \cap l_2\neq\emptyset\}$; Note that $I$ has
codimension $n-2$ in $F\times F$;\\
\indent $\Phi=p_*q^*$, homomorphism from either the cohomology groups
or the Chow groups of $X$ to those of $F$;\\
\indent $\Psi=q_*p^*$, homomorphism from either the cohomology groups
or the Chow groups of $F$ to those of $X$;\\
\indent $g$, the polarization on $F$ that comes from the Pl\"ucker
embedding, $g=\Phi(h^2)$;\\
\indent $F_x\subset F$, the subscheme of lines passing through $x\in
X$; it is a (2,3) complete intersection in $\PP(\mathcal{T}_{X,x})$
and
smooth for general $x$.\\
\indent $D_x\subset X$, the variety swept out by all lines through $x$; $D_x$ is a cone over $F_x$;\\
\indent $F_C\subset F$, the subscheme parameterizing lines meeting
a curve $C\subset X$;\\
\indent $S_C=q^{-1}(C)\subset P(X)$, note that there are natural
morphism $i_C=p|_{S_C}: S_C\rightarrow F_C\subset F$ and $q_0=q|_{S_C}:S_C\rightarrow C$;\\
\indent $g|_{S_C}=(i_C)^*g$, by abuse of notation, we still use $g$ to denote this class;\\
\indent $g'=(q_0)^*[pt]$, where $[pt]\in C$ is a closed point;\\
\indent $\Q[g,g']\subset \mathrm{H}^{*}(S_C,\Q)$, the subring
generated by $g$ and $g'$;\\
\indent $P_C=P(X)|_{S_C}$, to be more precise, we take the following
fiber product
$$
\xymatrix{
 P_C\ar[r]^{j_C}\ar[d]_{p_C} &P(X)\ar[d]^p\\
 S_C\ar[r]^{i_C} &F
}
$$
\indent $q_C=q\circ j_C$; \\
\indent $\Phi_C=(p_C)_*(q_C)^*$, homomorphism from either the cohomology
groups or the Chow groups of $X$ to those of $S_C$;\\
\indent $\Psi_C=(q_C)_*(p_C)^*$, homomorphism from either the cohomology
groups or the Chow groups of $S_C$ to those of $X$;\\
\indent $D_C\subset X$, the divisor swept out by all the lines
meeting
$C$; $D_C$ is linearly equivalent to $5\deg(C)h$; see Lemma \ref{lem dimension};\\
\indent Given a polarization $H$ on $Y$, we use
$\mathrm{H}^*(Y,\Z)_{prim}$ to denote
the primitive cohomology;\\
\indent Let $(Y,H)$ be a polarized variety, we define
$\mathrm{A}_*(Y)\subset\mathrm{CH}_*(Y)$ to be the subgroup of
degree 0 (with respect to $H$) cycles;\\
\indent For a vector bundle $\mathscr{E}$ on $Y$, we use
$\PP(\mathscr{E})$ to denote the geometric projectivization which
parameterizes all 1-dimensional linear subspaces of $\mathscr{E}$;
more generally, if $1\leq r_1<\cdots <r_k<\rk\mathscr{E}$ is an
increasing sequence of integers, we use $G(r_1,\ldots,
r_k,\mathscr{E})$ to denote the relative flag variety of subspaces
of $\mathscr{E}$ with corresponding ranks.

\section{The fundamental relations}

In this section we will establish a basic relation among
algebraic/topological cycles on $X$. To do that, we need the
following lemma which says that the space of lines on $X$ meeting a
given curve has the expected dimension.
\begin{lem}\label{lem dimension}
(i) $\dim F_x=n-3$ for all but finitely many points $x_i$, called
Eckardt points. For each $x_i$, we have $\dim F_{x_i}=n-2$.

(ii) Let $C\subset X$ be a smooth curve on $X$. Then $F_C$ if of
pure expected dimension $n-2$.

(iii) The divisor $D_C$ on $X$ is linearly equivalent to
$5\deg(C)h$.
\end{lem}

\begin{proof}
In \cite[Lemma 2.1]{cs}, it is shown that $\dim
F_x=n-3$ for a general point $x\in X$. By \cite[Corollary 2.2]{cs}, there are at most finitely
many points $x_i$, called Eckardt points (see \cite[Definition 2.3]{cs}).
This proves (i). Statement (ii) follows from (i) directly. Statement
(ii) further implies that $D_C$ is a divisor. Since $\Pic(X)=\Z h$,
the class of $D_C$ has to be a multiple of $h$. Let $l\subset X$ be
a general line, then the intersection number of $l$ and $D_C$ is
equal to the number of lines meeting both $C$ and $l$. It is shown
in \cite[Lemma 3.10]{shen} that the above intersection number is
$5\deg(C)$.
\end{proof}

\begin{thm}\label{fundamental relation}
Let $C\subset X$ be a smooth curve of degree $e$. We use $h$ to
denote the class of a hyperplane section on $X$, viewed as an
element either in the Chow group or the (co)homology group. Then the
following are true.

(1) Let $\gamma$ be a topological cycle of real dimension $n$ on
$X(\C)$. Then
$$2e[\gamma]+\Psi(\Phi([\gamma])\cdot F_C)=3e\deg(\gamma)h^{\frac{n}{2}}$$
in $\mathrm{H}_n(X)$, where $\deg(\gamma)=\gamma\cdot h^i$ if $n=2i$
and $\deg(\gamma)=0$ otherwise.

(2) Let $\gamma$ be an algebraic cycle on $X$ of dimension $r$ with
$2\leq r\leq n-2$. Then
$$2e\gamma+\Psi(\Phi(\gamma)\cdot F_C)=3e\deg(\gamma)h^{n-r},$$
in $\mathrm{CH}_r(X)$, where $\deg(\gamma)=\gamma\cdot h^r$.

(3) Let $\gamma$ be an algebraic cycle of dimension 1 on $X$. Then
$$2e\gamma+\Psi(\Phi(\gamma)\cdot F_C)+2\deg(\gamma)C=3e\deg(\gamma)h^{n-1}$$
in $\mathrm{CH}_1(X)$.
\end{thm}

\begin{rmk}
Statement (1) holds for other dimensional topological cycles too.
However, by the Lefschetz hyperplane theorem, it is only interesting
when $\gamma$ has dimension $n$.
\end{rmk}

\begin{cor}\label{cor fundamental relation}
Let $\gamma\in\mathrm{H}_n(X,\Z)$, $\fa\in\mathrm{CH}_1(X)$ and
$\fb\in\mathrm{CH}_l(X)$. Then the following are true

(i)
$2\deg(\fa)\gamma+\Psi(\Phi(\gamma)\cdot\Phi([\fa]))=3\deg(\fa)\deg(\gamma)h^{\frac{n}{2}}$
in $\mathrm{H}_n(X,\Z)$.

(ii) If $n-2\geq l\geq2$, then
$2\deg(\fa)\fb+\Psi(\Phi(\fa)\cdot\Phi(\fb))=3\deg(\fa)\deg(\fb)h^{n-l}$
in $\mathrm{CH}_l(X)$.

(iii) If $l=1$, then $2\deg(\fa)\fb +2\deg(\fb)\fa
+\Psi(\Phi(\fa)\cdot\Phi(\fb))=3\deg(\fa)\deg(\fb)h^{n-1}$ in
$\mathrm{CH}_1(X)$.
\end{cor}

\begin{lem}\label{lem linear span}
Let $\gamma_1$ and $\gamma_2$ be two disjoint irreducible
topological cycles on $\PP^{n+1}$ of real dimensions $r_1$ and
$r_2$. Let $L(\gamma_1,\gamma_2)$ be the set swept out by all
complex lines meeting both $\gamma_1$ and $\gamma_2$. Assume that
$\gamma_1$ and $\gamma_2$ are in general position and
$L(\gamma_1,\gamma_2)$ has expected dimension. Then we have
$$
L(\gamma_1,\gamma_2)=\deg(\gamma_1)\deg(\gamma_2)h^{n-\frac{r_1+r_2}{2}}
$$
in $\mathrm{H}_{r_1+r_2+2}(\PP^{n+1},\Z)$. Here
$\deg(\gamma)=\gamma\cdot h^{\frac{r}{2}}$ if $r=\dim_\R(\gamma)$ is
even and 0 otherwise. By convention $h^{k+\frac{1}{2}}=0$ for all
integer $k$.
\end{lem}

\begin{lem}\label{lem c1 product}
Assume that $\gamma_1$ and $\gamma_2$ are two topological manifolds
(resp. smooth projective varieties). Let
$f_i:\gamma_i\rightarrow\PP^{n+1}$, $i=1,2$, be two disjoint
irreducible topological (algebraic) cycles on $\PP^{n+1}$. Let
$\varphi:\gamma_1\times\gamma_2 \rightarrow G(2,n+2)$ be the map
(morphism) sending a pair of points $(x,y)$ to the unique complex
line connecting them. Let $\mathscr{E}_2$ be the canonical rank 2
quotient bundle on $G(2,n+2)$ and
$\mathscr{E}'=\varphi^*\mathscr{E}_2$ its pullback. Let
$p_i:\gamma_1\times\gamma_2 \rightarrow \gamma_i$, $i=1,2$, be the
two projections. Then
$$
c_1(\mathscr{E}')=p_1^*(f_1^*h) +p_2^*(f_2^*h)
$$
in $\mathrm{H}^2(\gamma_1\times\gamma_2,\Q)$ (resp.
$\mathrm{CH}^1(\gamma_1\times\gamma_2)$), where $h$ is the class of
a hyperplane. In other words, $c_1(\mathscr{E}')$ is the pull-back of the class $\pi_1^*h +\pi_2^* h$ via the natural map $f_1\times f_2: \gamma_1\times\gamma_2 \rightarrow \PP^{n+1}\times\PP^{n+1}$, where $\pi_i:\PP^{n+1}\times \PP^{n+1}\rightarrow \PP^{n+1}$ are the natural projections.
\end{lem}

Since the cohomology (Chow ring) of $\PP^{n+1}$ is generated by the
class of a hyperplane, both of the above lemmas can be reduced to
the case where both $\gamma_1$ and $\gamma_2$ are complex linear
subspaces. The proofs are left to the reader.

\begin{proof}[Proof of Theorem \ref{fundamental relation}]
To prove (1), we first reduce to the case $\gamma=f_*[M]$, where $f:
M\rightarrow X$ be a continuous map from an $n$-dimensional
topological manifold $M$ to $X$ such that (i) $C$ does not meet
$f(M)$; (ii) $f_*M$ intersects $D_C$ transversally and
$\Delta_0:=\Set{t\in M}{f(t)\in D_C}$ is of pure expected real
dimension $\dim M-2$. The reason why we can do the above reduction
can be seen as follows. First, $\gamma$ is always of the form of
linear combinations of cycles of the form $f_*[M]$. If we can prove
the statement (1) for each term in such a the linear combination,
then we prove (1) for $\gamma$. To get (i) and (ii), we note that
$C$ is a curve and $D_C$ is a divisor and hence we can always move
the cycle $\gamma$ to a cycle $\gamma'$ such that (i) and (ii) hold.

Now we assume the situation after reduction. For each pair of points
$(t,x)\in M\times C$, there is a unique (complex projective) line
passing through $f(t)$ and $x$. This gives a continuous map
$i_0:M\times C\rightarrow \mathrm{G}(2,n+2)$. We get the following
diagram
\begin{equation}\label{topology basic diagram}
\xymatrix{
  D_1\cup D_2\cup D\ar[r]\ar[d] &Y\ar[r]\ar[d] &X\ar[d]^{i_X}\\
  P\ar[r]^{i'_0\qquad}\ar[d]_\pi &\mathrm{G}(1,2,n+2)\ar[d]_{\tilde{\pi}}
  \ar[r]^{\quad\tilde{\varphi}} &\PP^{n+1}\\
  M\times C\ar[r]^{i_0\quad} &\mathrm{G}(2,n+2) &
}
\end{equation}
where all the squares are fiber products; $D_1$ and $D_2$ are
sections of $\pi$ corresponding to the points on $f(M)$ and $C$
respectively; $D$ is the set of the third points of the intersection of the lines with
$X$. Let $\pi':D\rightarrow M\times C$ be the restriction of $\pi$
to $D$. Let $\Delta\subset M\times C$ be the closed subset of points
$(t,x)$ such that the line through $f(t)$ and $x$ is contained in
$X$. By definition, $\Delta_0$ is the image of $\Delta$ under the
projection to the first factor $M$ and hence
$\dim\Delta=\dim\Delta_0$. Then $\pi'$ is one-to-one away from
$\Delta$ while over $\Delta$ it is an $S^2$-bundle with trivial
Euler class (note that by taking the point in which the line meets
$C$, we have a section of this bundle). Consider the following
diagram
\begin{equation}\label{topology blow up}
\xymatrix{
 \PP^1=S^2\ar[r] &E\ar[r]\ar[d] &D\ar[d]^{\pi'}\\
     &\Delta\ar[r] &M\times C
}
\end{equation}
For a general point $z=(t_0,x_0)\in\Delta$, we use $E_{z}$ to denote
the fiber of $E\rightarrow \Delta$ at the point $z$.\\

\textsl{Claim 1}: $E\cdot E_z=-1$.
\begin{proof}[Proof of Claim 1]
Let $U_0\subset C$ be a small 2-dimensional (meaning real dimension)
disc centered at the point $x_0\in C$. Let $U_z=\{t_0\}\times U_0
\subset M\times C$ be the corresponding small disc centered at $z$.
By the assumption that $z\in\Delta$ is general, we see that $U_z$
meets $\Delta$ transversally at the point $z$. Then we easily see
that $\pi'^*U_z=\tilde{U}_z + E_z$, where
$\tilde{U}_z=\pi'^{-1}(U_z\backslash\{z\})\cup\{y_0\in E_z\}$. The
existence of the point $y_0$ can be seen as follows. Let
$f_0:C\dashrightarrow X$ be the rational map defined by sending
$x\in C\subset X$ to the third point of the intersection of $X$ with
the line connecting $f(t_0)$ and $x$. Since $C$ is a smooth curve,
this rational map extends to a morphism $f_0:C\rightarrow X$. Then
one easily sees that $y_0=f_0(x_0)$. This implies that
$\tilde{U}_z\cdot E=1$. The projection formula gives
$$
(E_z+\tilde{U}_z)\cdot E = \pi'^*U_z\cdot E = U_z\cdot
\pi_*E=U_z\cdot 0=0.
$$
It follows that $E_z\cdot E=-1$.
\end{proof}

Let $g:D\rightarrow X$ and $g_i:D_i\rightarrow X$, $i=1,2$, be the
natural maps. Let $h$ be the (co)homology class of a hyperplane.
\begin{lem}\label{topology key lemma}
Let $d=\deg(f_*[M])$. The following are true.\\
(i) $g_*[D]=ed\,h^{\frac{n}{2}-1}$ in $\mathrm{H}_{n+2}(X)$;\\
(ii) $g_*(h|_D)=ed\,h^{\frac{n}{2}}$ in $\mathrm{H}_n(X)$;\\
(iii) $h|_D+[E]=2\pi'^*c_1(\mathscr{E})$, where
$\mathscr{E}=i_0^*\mathscr{E}_2$ is the pull back of the canonical
rank 2 quotient bundle $\mathscr{E}_2$ on
$\mathrm{G}(2,n+2)$;\\
(iv) $c_1(\mathscr{E})=p_1^*(h|_M)+p_2^*(h|_C)$, where $p_1:M\times
C\rightarrow M$ and $p_2:M\times C\rightarrow C$ are the natural
projections;\\
(v) $g_*(\pi'^*c_1(\mathscr{E}))=2ed\,h^{\frac{n}{2}}-e\,f_*[M]$.
\end{lem}
\begin{proof}[Proof of Lemma \ref{topology key lemma}]
For (i), we first note that the class of the image of
$[D_1]+[D_2]+[D]$ on $X$ is the restriction of $[P]$ viewed as a
class on $\PP^{n+1}$. As a class on $\PP^{n+1}$,
$[P]=ed\,h^{\frac{n}{2}-1}$ by Lemma \ref{lem linear span}. However,
if we identify $D_i$, $i=1,2$, with $M\times C$, then the natural
map $D_i\rightarrow X$ contacts either the factor $M$ or the factor
$C$. Hence the classes of $(g_1)_*[D_1]$ and $(g_2)_*[D_2]$ are all
zero on $X$. This proves (i).

For (ii), we note that by the projection formula, we have
$g_*(h|_D)=g_*[D]\cdot h$. Hence (ii) follows from (i). The identity
in (iv) follows from Lemma \ref{lem c1 product}.

To prove (iii), we first note that $h|_D+[E]$ restricts to zero on
the fibers $E_z$ of $E\rightarrow\Delta$ since $h\cdot E_z=1$ and
$E\cdot E_z=-1$ (Claim 1). By the Leray-Hirsch theorem, this implies
that $h|_D+[E]=\pi'^*\mathfrak{a}$ where $\mathfrak{a}$ is a
homology class on $M\times C$. Applying $\pi'_\ast$ to the above
equation, we get $\fa=\pi'_\ast (h|_D)$. Note that on $P$, we have
$D_1+D_2+D=3h$, where, by abuse of notation, we still use $h$ to
denote the pullback of the class of a hyperplane to $P$. Hence
\begin{equation}\label{identity 1}
h\cdot D_1+h\cdot D_2+ h\cdot D=3h^2
\end{equation}
Since $P$ is the projectivization of $\mathscr{E}$ and $h$ is the
first Chern class of the relative $\calO(1)$-bundle, we have
$h^2-\pi^*c_1(\mathscr{E})h+\pi^*c_2(\mathscr{E})=0$. Applying
$\pi_*$ to \eqref{identity 1}, we get
\begin{align*}
\pi_*(h\cdot D)+\pi_*(h\cdot D_1)+\pi_*(h\cdot D) &=3\pi_*(h^2)\\
  &=3\pi_*(\pi^*c_1(\mathscr{E}) \cdot h-\pi^*c_2(\mathscr{E}))\\
  &=3c_1(\mathscr{E})\cdot\pi_*h=3c_1(\mathscr{E})
\end{align*}
We also easily get that $\pi_*(h\cdot D_1)=p_1^*f^*h$ and
$\pi_*(h\cdot D_2)=p_2^*(h|_C)$. Combine this with (iv), we get
$$
\fa=\pi'_*(h|_D)=\pi_*(h\cdot D)=2c_1(\mathscr{E}).
$$

To prove (v), we note that for any class $\fa$ on $M\times C$, if we
pull back $\fa$ to $D_1+D_2+D$ and then push forward to $X$, then
what we get is the same class as we pull back $\fa$ to $P$, then
push forward to $\PP^{n+1}$ and then restrict to $X$. As a result,
we always get a class coming from $\PP^{n+1}$. When we take
$\fa=c_1(\mathscr{E})$, we get that the class of $\pi^*\fa$ on
$\PP^{n+1}$ is equal to $2ed\,h^{\frac{n}{2}}$. This can be seen as
follows (using notation from Lemma \ref{lem linear span}).
\begin{align*}
 \tilde{\varphi}_*(i'_0)_*\pi^*c_1(\mathscr{E}) &=
 \tilde{\varphi}_*(i'_0)_*\pi^* (p_1^*f^*h+ p_2^*h|_C)\\
 &= L(f_*M\cdot h, C) + L(M,C\cdot h)\\
 &= ed\,h^{n-\frac{(n-2)+2}{2}} + ed\,h^{n-\frac{n+0}{2}}\\
 &=2ed\,h^{\frac{n}{2}}.
\end{align*}
Hence we have the following equalities.
\begin{align*}
g_*(\pi'^*c_1(\mathscr{E})) &=2ed\,h^{\frac{n}{2}}
-({g_1}_*(\pi^*c_1(\mathscr{E})|_{D_1}) +
{g_2}_*(\pi^*c_1(\mathscr{E})|_{D_2}))\\
  &=2ed\,h^{\frac{n}{2}}-{g_1}_*(\pi^*c_1(\mathscr{E})|_{D_1})\\
  &=2ed\,h^{\frac{n}{2}}-e\,f_*[M]
\end{align*}
Here the second equality uses the fact that $g_2:D_2\rightarrow
C\subset X$ contracts the factor $M$ which has dimension $n\geq 3$.
This finishes the proof of the lemma.
\end{proof}

In the above lemma, we apply $g_*$ to (iii) and then take (ii) and
(v) into account.
$$
g_*[E] =2g_*(\pi'^*c_1(\mathscr{E})) -
g_*(h|_D)=4ed\,h^{\frac{n}{2}}-2e\,f_*[M] -ed\,h^{\frac{n}{2}} =
3ed\,h^{\frac{n}{2}}-2e\,f_*[M]
$$
Note that $g_*[E]=\Psi(\Phi(f_*[M])\cdot F_C)$. This way, we easily
deduce the statement (1) of Theorem \ref{fundamental relation} for
$\gamma=f_*[M]$ and the curve $C$.

Now we start to prove (2) and (3) of Theorem \ref{fundamental
relation}. The basic strategy is the same as above. The only
difference is that we need to consider the more delicate rational
equivalence rather than homological equivalence. By linearity of the
equalities in (2) and (3) and the Chow moving lemma, we reduce to
the case $\gamma=f_*M$ where $f:M\rightarrow X$ is a morphism from a
smooth projective variety $M$ to $X$ such that (i) $f$ is birational
onto its image; (ii) $f(M)$ does not meet $C$ and $f_*M$ intersects
$D_C$ transversally. Let $r$ be the dimension of $M$. As before we
construct the natural morphism $i_0:M\times C\rightarrow
\mathrm{G}(2,n+2)$. Let $\mathscr{E}$ be the pull back of the
canonical rank 2 quotient bundle to $M\times C$. In this way, we get
the following diagram as before.
\begin{equation}\label{cycle basic diagram}
\xymatrix{
  D_1\cup D_2\cup D\ar[r]\ar[d] &Y\ar[r]\ar[d] &X\ar[d]\\
  P\ar[r]^{i'_0\qquad}\ar[d]_\pi &\mathrm{G}(1,2,n+2)\ar[d]^{\tilde{\pi}}
  \ar[r]^{\qquad\tilde{\varphi}} &\PP^{n+1}\\
  M\times C\ar[r]^{i_0\quad} &\mathrm{G}(2,n+2) &
}
\end{equation}
Here all the squares are fiber products; $D_1$ and $D_2$ are
sections of $\pi$ and contracted to $f(M)$ and $C$ via
$g_1=\tilde{\varphi}\circ i'_0|_{D_1}:D_1\rightarrow X$ and
$g_2=\tilde{\varphi}\circ i'_0|_{D_2}:D_2\rightarrow X$
respectively; $D$ corresponds to the third point of the intersection
of the lines with $X$. Let $\pi':D\rightarrow M\times C$ be the
restriction of $\pi$ to $D$. Then $\pi'$ is a birational morphism.
Let $\Delta\subset M\times C$ be the subvariety that consists of
points $(t,x)$ such that the line through $f(t)$ and $x$ is
contained in $X$. By the assumption that $f_*M$ intersects $D_C$
transversally, we conclude that $\Delta$ is generically smooth. Let
$\Delta^{sing}$ be the singular locus of $\Delta$. Let
$E=(\pi')^{-1}\Delta\subset D$ and we have
$E\cong\PP(\mathscr{E}|_\Delta)$.

\begin{lem}\label{lem blow up}
Away from $\Delta^{sing}$, $\pi'$ is the blow-up along $\Delta$.
\end{lem}
\begin{proof}
Let $\mathscr{E}_2$ be the canonical rank 2 quotient bundle on
$G(2,n+2)$ and $\mathscr{E}=i_0^*\mathscr{E}_2$. The point
$i_0(t,x)\in G(2,n+2)$ represents the line connecting $f(t)$ and
$x$, which is also naturally $\PP(\mathscr{E}^\vee_{(t,x)})$. Hence
we have canonical homomorphisms $\mathcal{L}_{f(t)}\hookrightarrow
\mathscr{E}^\vee_{(t,x)}$ and $\mathcal{L}_{x}\hookrightarrow
\mathscr{E}^\vee_{(t,x)}$, where
$\mathcal{L}\cong\calO_{\PP^{n+1}}(-1)$ is the tautological line
bundle on $\PP^{n+1}$. And since $f(t)\neq x$, we have the natural
identification
$$
\mathscr{E}^\vee_{(t,x)} = \mathcal{L}_{f(t)} \oplus
\mathcal{L}_{x}.
$$
As $(t,x)\in M\times C$ varies, this identification glues and gives
a canonical isomorphism
$$
\mathscr{E}^\vee \cong p_1^*f^*\mathcal{L} \oplus
p_2^*(\mathcal{L}|_C).
$$
Assume that $\phi(X_0,\ldots,X_{n+1})\in
\mathrm{H}^0(\PP^{n+1},\calO(3))$ is the homogeneous polynomial
defining $X$. Since we have the canonical identification
$\tilde{\pi}_* \tilde{\varphi}^* \calO(3)
=\mathrm{Sym}^3(\mathscr{E}_2)$, the section $\phi$ induces a global
section $\tilde{\phi}$ of $\Sym^3(\mathscr{E}_2)$ and $F\subset
G(2,n+2)$ is defined by $\tilde{\phi}=0$; see \cite{ak}. Let
${\phi}'=i_0^*\tilde{\phi}$ be the induced global section of
$\Sym^3\mathscr{E}$. It follows from the expression of
$\mathscr{E}^\vee$, we get
$$
\mathscr{E} \cong p_1^*f^*\calO_X(1) \oplus p_2^*\calO_C(1)
$$
and hence
$$
\Sym^3\mathscr{E} = p_1^*f^*\calO_X(3) \oplus \mathscr{E}\otimes
p_1^*f^*\calO_X(1)\otimes p_2^*\calO_C(1) \oplus p_2^*\calO_C(3).
$$
Accordingly, we can write $\phi'=\phi'_{3,0}+\phi'_\mathrm{mid}
+\phi'_{0,3}$. Since $f(M)\subset X$ and $\phi$ vanishes on $X$, we
conclude that $\phi'_{3,0}=0$. Similarly, we have $\phi'_{0,3}=0$.
It follows that
$$
\phi'=\phi'_\mathrm{mid}\in \mathrm{H}^0(M\times C,
\mathscr{E}\otimes p_1^*f^*\calO_X(1)\otimes p_2^*\calO_C(1)).
$$
The vanishing $\phi'=0$ at a point $(t,x)$ exactly means that the
line connecting $f(t)$ and $x$ is contained in $X$. Hence
$\Delta\subset M\times C$ is defined by $\phi'=0$. In particular,
$\Delta\subset M\times C$ is a local complete intersection and
$$
\mathscr{N}_{\Delta/M\times C}\cong \mathscr{E}\otimes
p_1^*f^*\calO_X(1)\otimes p_2^*\calO_C(1)|_\Delta.
$$
The section $\phi'$ can be viewed as a homomorphism
$$
\phi': p_1^*f^*\calO_X(-1)\otimes p_2^*\calO_C(-1) \longrightarrow
\mathscr{E}.
$$
This gives a section $s:M\times C\backslash\Delta\rightarrow P$ of
$\pi$ such that the image of $s$ is the open subset $D\backslash E$
of $D$. Let $\sigma:\widetilde{M\times C}\rightarrow M\times C$ be
the blow-up of $M\times C$ along $\Delta$ and
$E_0=\sigma^{-1}\Delta$. As a consequence of the above description
of $\mathscr{N}_{\Delta/M\times C}$, we see that $\phi'$ extends to
$\widetilde{M\times C}\backslash \sigma^{-1}\Delta^{sing}$ and gives
a rank one subbundle
$$
\tilde{\phi}':\mathscr{L}\hookrightarrow \sigma^*\mathscr{E}
$$
where $\mathscr{L}=\sigma^*(p_1^*f^*\calO_X(-1)\otimes
p_2^*\calO_C(-1))\otimes \calO(E_0)$ is a line bundle on
$\widetilde{M\times C}\backslash \sigma^{-1}\Delta^{sing}$. This
further gives a morphism
$$
\tilde{s}:\widetilde{M\times C}\backslash \sigma^{-1}\Delta^{sing}
\hookrightarrow P
$$
such that $\sigma=\pi\circ\tilde{s}$. Note that $\tilde{s}$ extends
$s$. Hence the image of $\tilde{s}$ is contained in $D$ since that
of $s$ is. Since $\tilde{E}\rightarrow\Delta$ and $E\rightarrow
\Delta$ are all $\PP^1$-bundles, we conclude that The image of
$\tilde{s}$ is exactly $D\backslash (\pi')^{-1}\Delta^{sing}$. This
implies that
$$
\tilde{s}:\widetilde{M\times C}\backslash \sigma^{-1}\Delta^{sing}
\longrightarrow D\backslash (\pi')^{-1}\Delta^{sing}
$$
is an isomorphism, since they are the same locally closed subvariety
of $P$.
\end{proof}

Back to the proof of the theorem. We know from the above lemma that
the possible singularities of $D$ can only appear over
$\Delta^{sing}$. Take a resolution of singularities
$r':\tilde{D}\rightarrow D$. Let $E_i$ be the exceptional divisors
of $r'$. We use $E'$ to denote the strict transform of $E$ in
$\tilde{D}$. We still use $h$ to denote the class of a hyperplane in
$\PP^{n+1}$. Let $g:\tilde{D}\to X$ be the natural morphism.
\begin{lem}\label{cycle key lemma}
Let $d=\deg(f_*[M])$, then the following are true.\\
(i) $g_*\tilde{D}=ed\,h^{n-r-1}$ in $\mathrm{CH}_{r+1}(X)$;\\
(ii) $g_*(h|_{\tilde{D}})=ed\,h^{n-r}$;\\
(iii) $h|_{\tilde{D}}+E'+\sum a_iE_i=2(\pi'\circ
r')^*c_1(\mathscr{E})$,
for some $a_i\in\Z$;\\
(iv) $c_1(\mathscr{E})=p_1^*(f^*h)+p_2^*(h|_C)$;\\
(v) If $\dim M\geq 2$, then
$g_*(r'^*\pi'^*c_1(\mathscr{E}))=-e\,f_*M+2ed\,h^{n-r}$ in $\mathrm{CH}_r(X)$;\\
(vi) If $M$ is a curve, then
$g_*(r'^*\pi'^*c_1(\mathscr{E}))=ed\,h^{n-1}-e\,f_*M-dC$ in
$\mathrm{CH}_1(X)$.
\end{lem}
\begin{proof}[Proof of Lemma \ref{cycle key lemma}]
The proof is similar to that of Lemma \ref{topology key lemma}. We
note that the push forward of $D_1+D_2+D$ to $X$ is a class coming
from $\PP^{n+1}$ because it is the restriction of the class of the
image of $P$ in $\PP^{n+1}$. By Lemma \ref{lem linear span}, the
class of $P$ in $\PP^{n+1}$ is equal to $ed\,h^{n-r-1}$. Since $D_1$
and $D_2$ are contracted to smaller dimensions via $g_1$ and $g_2$
respectively, we get (i).

Conclusion (ii) follows from the projection formula as before. (iv)
follows from Lemma \ref{lem c1 product}.

For (iii), note that on $D-\pi'^{-1}(\Delta^{sing})$, the class of
$h+E$ comes from a class $\fa$ on $M\times C-\Delta^{sing}$ via pull
back by $\pi'$. Since the codimension of $\Delta^{sing}$ in $M\times
C$ is at least 3, we know that the divisor class group of $M\times
C-\Delta^{sing}$ is the same as that of $M\times C$. Hence we can
view $\fa$ as divisor class on $M\times C$. The equality
$$
 h+E=\pi'^*\fa,\quad\text{on }D\backslash \pi'^{-1}(\Delta^{sing})
$$
gives an equation
$$
r'^*(h+E)+ \sum a_iE_i= r'^*\pi'^*\fa
$$
for some $a_i\in\Z$. By a very similar argument as before we know
that the class $\fa$ is $2c_1(\mathscr{E})$.

To prove (v) and (vi), we do the following explicit calculation.
\begin{align*}
g_*r'^*\pi'^*c_1(\mathscr{E}) &= g_*r'^*\pi'^*(p_1^*f^*h +
p_2^*h|_C),\qquad\text{by (iv)}\\
 &= \tilde{g}_* \pi^*(p_1^*f^*h + p_2^*h|_C)|_{D\cup D_1\cup D_2} -
 g_{1*} \pi^*(p_1^*f^*h + p_2^*h|_C)|_{D_1}\\
 &\quad - g_{2*}\pi^*(p_1^*f^*h + p_2^*h|_C),\qquad\tilde{g}:D\cup
 D_1\cup D_2\rightarrow X\\
&= \tilde{\varphi}_*(i'_0)_*\pi^*(p_1^*f^*h + p_2^*h|_C)|_X
-g_{1*}(p_2^* h|_C) -g_{2*}(p_1^*f^*h)\\
&\qquad (\text{here we identify }D_1\text{ and }D_2 \text{ with
}M\times C)\\
&= L(f_*M\cdot h, C)|_X + L(f_*M,C\cdot h)|_X -e\,f_*M
-g_{2*}(p_1^*f^*h)\\
&= 2ed\,h^{n-r} -e\,f_*M -g_{2*}(p_1^*f^*h)
\end{align*}
Note that $g_{2*}(p_1^*f^*h)$ is supported on the curve $C$. If
$\dim M\geq 2$, then $g_{2*}(p_1^*f^*h)=0$; if $\dim M=1$, then
$g_{2*}(p_1^*f^*h)=\deg(f_*M)C$. Hence (v) and (vi) follow from the
above computation.
\end{proof}

Now we come back the the proof of the theorem. In the above lemma,
we first apply $g_*$ to (iii) and note that all the $E_i$'s map to
zero. We also easily see that $g_*E'=\Psi(\Phi(f_*M)\cdot F_C)$.
Combine all these with (v) (or (vi) in the case of 1-cycles), we get
the conclusion (2) and (3) for $\gamma=f_*M$. By linearity, the
conclusions hold for any given $\gamma$.
\end{proof}

\section{The action of incidence correspondence}
Let $C\subset X$ be a smooth rational curve of degree $e$ on $X$. We
also assume that $C$ is general, meaning that it comes from a
non-empty open subset of the corresponding component of the Hilbert
scheme. Let $S_C=q^{-1}C$ be the inverse image of $C$ under the
morphism $q:P(X)\rightarrow X$. The points on $S_C$ can be described as
$$
S_C=\{([l],x)\in F\times C: x\in l\}.
$$
Let $q_0=q|_{S_C}:S_C\rightarrow C$ be the natural morphism.
\begin{lem}\label{lem smoothness of Sc}
If $C$ is general, then $S_C$ is smooth of dimension $n-2$.
\end{lem}

We will prove this lemma later. In \cite[Lemma 3.4]{shen}, we show that for a general rational curve
$C$ of degree $e\geq 2$, there exist $n_e=\frac{5e(e-3)}{2}+6$
secant lines of $C$, i.e. lines meeting $C$ twice. Let $L_{C,i}$,
$i=1,\ldots,n_e$, be all the secant lines of $C$ and $[L_{C,i}]\in
F$ the corresponding points on the variety of lines. Above each
point $[L_{C,i}]$, we have a pair of points
$\{([L_{C,i}],y_i),([L_{C,i}],z_i)\}$ on $S_C$, where $y_i$ and
$z_i$ are the two points in which $C$ intersects $L_{C,i}$. Then
Lemma \ref{lem smoothness of Sc} implies that $[L_{C,i}]$'s are the
only singular points of $F_C$ and $p|_{S_C}:S_C\rightarrow F_C$ is
the normalization and also a desingularization since $S_C$ is
smooth.

From now on, we make the assumption that $e\geq 2$, unless otherwise
stated.
\begin{defn}
We say that a correspondence $\Gamma\subset Y\times Y$ is
generically defined by $y\mapsto\sum y_i$ if $\Gamma$ is the closure
of the graph of this multi-valued map.
\end{defn}

For a general point $[l]\in F_C$, let $x=C\cap l$ be the
intersection point. By $[l]\mapsto ([l],x)\in S_C$, we view $[l]$ as
a point on $S_C$.
\begin{lem}\label{lem number of incidence lines}
There exist $5e-5$ lines
$l_1,l_2,\ldots,l_{5e-5}$ meeting both $l$ and $C$ in points
different from $x$.
\end{lem}
By abuse of language, these lines are called
\textit{secant lines} of the pair $(l,C)$; see \cite[Definition 3.1]{shen}. Each point $[l_i]$ can
again be viewed as a point $([l_i],x_i)$ on $S_C$, where
$x_i=l_i\cap C$. A line meeting two disjoint curves $C_1,C_2\subset
X$ will be called a \textit{incidence line} of $C_1$ and $C_2$ (note that they
are also called secant lines in \cite{shen}).

\begin{proof}[Proof of Lemma \ref{lem number of incidence lines}]
Note that if $\tilde{l}$ is a general line, the number of incidence lines
of $\tilde{l}$ and $C$ is $5e$; this follows from a degree
computation using (3) of Theorem \ref{fundamental relation} or
\cite[Lemma 3.10]{shen}. When $\tilde{l}$ specializes to $l$, five
of these incidence lines of $\tilde{l}$ and $C$ will specialize to
five lines passing through $x$ and those five lines are not counted
as the secant lines of the pair $(l,C)$. We will describe this
specialization in more details later. Hence the number of secant
lines of the pair $(l,C)$ is $5e-5$.
\end{proof}

\begin{defn}
Let the incidence correspondence $I_C\subset S_C\times S_C$ be
generically defined by
\begin{equation}
I_C:([l],x)\mapsto \sum_{i=1}^{5e-5}([l_i],x_i),
\end{equation}
where $l$ is a line meeting $C$ and $l_i$ are the secant lines of
the pair $(l,C)$. Let $\sigma$
denote the action of $I_C$ on either the cohomology groups or the
Chow groups of $S_C$. On $F$, we have the incidence correspondence
$I=\{([l_1],[l_2])\in F\times F: l_1\cap l_2\neq \emptyset\}\subset
F\times F$. This induces a correspondence
$$
I'_C=(i_C\times i_C)^*I\in\mathrm{CH}_{n-2}(S_C\times S_C),
$$
where $i_C=p|_{S_C}:S_C\to F$ is the natural morphism.
\end{defn}

\begin{rmk}
Note that $(i_C\times i_C)^{-1}I$ has a component, namely
$S_C\times_C S_C$, which has dimension bigger than expected. The
correspondence $I_C\subset (i_C\times i_C)^{-1}I$ is a component of
expected dimension. It turns out later that one key ingredient to
understand the action of $\sigma$ is the difference between $I_C$
and $I'_C$.
\end{rmk}

Note that on $F$, we have a natural polarization $g$ given by the
Pl\"ucker embedding of $\mathrm{G}(2,n+2)$. It can also be written
as $g=\Phi(h^2)$. We fix $g|_{S_C}=(i_C)^*g$ as the polarization of
$S_C$. By abuse of notation, we still use $g$ to denote its
restriction to $S_C$. Recall that $S_C$ admits a natural morphism
$q_0=q|_{S_C}:S_C\rightarrow C$. This gives an extra class
$g'=q_0^*[pt]$ on $S_C$. Note that $g'$ is never ample. The
following is the main result of this section.

\begin{thm}\label{sigma identity}
Let $C\subset X$ be a general rational curve of degree $e\geq 2$ and
$\sigma$ the action of incidence correspondence as above. Then the
following are
true.

(1) Let $\fa$ be a topological cycle of odd dimension on $S_C$. Then
\begin{equation*}
\sigma(\fa)=\Phi_C(\Psi_C(\fa))+\fa
\end{equation*}
in $\mathrm{H}^*(X,\Z)$.

(2) Let $\fa$ be a topological cycle of dimension $2m$ or an
algebraic cycle of dimension $m$ on $S_C$, then in either the
cohomology $\mathrm{H}^{2n-2m-4}(X,\Z)$ or Chow group
$\mathrm{CH}_m(X)$, we have
$$
\sigma(\fa)=\Phi_C(\Psi_C(\fa))+\fa +\mathrm{const.}
$$
where the constant only depends on the intersection numbers
$\fa\cdot g^m$ and $\fa\cdot g'g^{m-1}$; the constant is zero if
$\fa\cdot g^m = 0$ and $\fa\cdot g'g^{m-1}=0$.
\end{thm}

Before giving a proof of the above theorem, we give some technical
constructions related to the geometry of $S_C$.

\subsection{The normal bundle of a line meeting $C$}

Recall from \cite[\S 6]{cg} that
for any line $l$ on $X$, we have either
$$
\mathscr{N}_{l/X}\cong\calO(1)^{n-3}\oplus \calO^2,
$$
in which case $l$ is said to be of first type, or
$$
\mathscr{N}_{l/X}\cong\calO(1)^{n-2}\oplus \calO(-1),
$$
in which case $l$ is said to be of second type. A general line is of
first type.

\begin{defn}\label{defn positive part}
Let $\mathscr{G}$ be a vector bundle on $\PP^1$. Then we have a
decomposition $\mathscr{G}=\bigoplus_{i=1}^{\rk \mathscr{G}}
\calO(a_i)$, $a_i\in\Z$. Then we define the positive part of
$\mathscr{G}$ to be
$$
\mathrm{Pos}(\mathscr{G}) =\bigoplus_{a_i>0}\calO(a_i).
$$
Similarly, we define the nonnegative part of $\mathscr{G}$ to be
$$
\mathrm{NN}(\mathscr{G}) =\bigoplus_{a_i\geq 0}\calO(a_i).
$$
\end{defn}

\begin{lem}\label{lem transversal l and C}
Assume that $C\subset X$ is general and $l$ is a line meeting $C$ in
a point $x$. Then $l$ meets $C$ transversally at the point $x$.
\end{lem}
\begin{proof}
If $l$ is tangent to $C$ at the point $x$, then $l$ is a secant line
of $C$. When $C$ is general, a secant line meets $C$ in two distinct
points. This can be seen as follows. We call a secant line of $C$ \textit{simple} if it meets $C$ in
two distinct points.
If one of the secant lines $l$ of $C$ is tangent to $C$ at a point $x$, then
we can make a small deformation of $C$ such that $l$ deforms to a simple one. Note here we use the assumption that
the degree of $C$ is at least 2. At the same time, any another simple
secant line will survive under small deformations. This shows that all secant lines of a general curve $C$
are simple.
\end{proof}

\begin{defn}\label{defn N_0}
Let $([l],x)\in S_C$ be a point where $l$ is a line on $X$ and $x$
is a point in which $l$ meets $C$. Since $C$ is general, $l$ meets
$C$ transversally at the point $x$. We define
$\mathscr{N}_{l/X}\langle \mathcal{T}_{C,x}\rangle
\hookrightarrow\mathscr{N}_{l/X}$ to be the subsheaf of sections $s$
such that $s(x)$ is in the direction of $\mathcal{T}_{C,x}$. Or
equivalently, $\mathscr{N}_{l/X}\langle \mathcal{T}_{C,x}\rangle$
fits into the following short exact sequence
$$
\xymatrix{
 0\ar[r] &\mathscr{N}_{l/X}\langle \mathcal{T}_{C,x} \rangle \ar[rr] &&\mathscr{N}_{l/X}\ar[rr]
 &&\frac{\mathcal{T}_{X,x}}{\mathcal{T}_{l,x}\oplus
 \mathcal{T}_{C,x}} \ar[r] &0
}
$$
\end{defn}

The tangent space of $S_C$ at a
point $([l],x)$ is canonically identified as
$$
\mathcal{T}_{S_C,([l],x)} = \mathrm{H}^0 (l,\mathscr{N}_{l/X}\langle
\mathcal{T}_{C,x}\rangle)
$$
See \cite[\S II.1]{kollar}. It follows from the short exact sequence
in Definition \ref{defn N_0} that
$$
\chi(\mathscr{N}_{l/X}\langle \mathcal{T}_{C,x}\rangle)=
\chi(\mathscr{N}_{l/X})-(n-2)=n-2=\dim S_C
$$
where we use the fact that $\dim S_C=n-2$ which is a consequence of
Lemma \ref{lem dimension}. Hence $\dim
\mathcal{T}_{S_C,([l],x)}=\dim S_C$ if and only if
$\mathrm{H}^1(l,\mathscr{N}_{l/X}\langle
\mathcal{T}_{C,x}\rangle)=0$. Equivalently, we have the following
\begin{lem}\label{lem smooth condition}
The variety $S_C$ is smooth at
$([l],x)$ if and only if $h^1(\mathscr{N}_{l/X}\langle
\mathcal{T}_{C,x}\rangle)=0$.
\end{lem}

\begin{prop}\label{prop cases of N_0}
The splitting of $\mathscr{N}_{l/X}\langle \mathcal{T}_{C,x}\rangle$
has the following possibilities.

(i) If $l$ is of first type and the image of $\mathcal{T}_{C,x}$ in
$\mathscr{N}_{l/X,x}$ is not contained in
$\mathrm{Pos}(\mathscr{N}_{l/X})_x$, then
$$
\mathscr{N}_{l/X}\langle \mathcal{T}_{C,x}\rangle \cong
\calO^{n-2}\oplus\calO(-1).
$$

(ii) If $l$ is of first type and the image of $\mathcal{T}_{C,x}$ in
$\mathscr{N}_{l/X,x}$ is contained in
$\mathrm{Pos}(\mathscr{N}_{l/X})_x$, then
$$
\mathscr{N}_{l/X}\langle \mathcal{T}_{C,x}\rangle \cong
\calO(1)\oplus \calO^{n-4}\oplus\calO(-1)^2.
$$

(iii) If $l$ is of second type and the image of $\mathcal{T}_{C,x}$
in $\mathscr{N}_{l/X,x}$ is not contained in
$\mathrm{Pos}(\mathscr{N}_{l/X})_x$, then
$$
\mathscr{N}_{l/X}\langle \mathcal{T}_{C,x}\rangle \cong
\calO^{n-2}\oplus\calO(-1).
$$

(iv) If $l$ is of second type and the image of $\mathcal{T}_{C,x}$
in $\mathscr{N}_{l/X,x}$ is contained in
$\mathrm{Pos}(\mathscr{N}_{l/X})_x$, then
$$
\mathscr{N}_{l/X}\langle \mathcal{T}_{C,x}\rangle \cong
\calO(1)\oplus\calO^{n-3}\oplus\calO(-2).
$$
\end{prop}
\begin{proof}Let $r_+=\rk(\mathrm{Pos}(\mathscr{N}_{l/X}))$
and
$a=\dim_{\C}\mathrm{Im}\{\mathrm{Pos}(\mathscr{N}_{l/X})\rightarrow
\frac{\mathcal{T}_{X,x}}{\mathcal{T}_{l,x}\oplus
\mathcal{T}_{C,x}}\}$. Consider the following commutative diagram
$$
\xymatrix{
 0\ar[r] &\mathscr{G} \ar[rr] &&\mathscr{Q} \ar[rr] && \C^{n-2-a}\ar[r] &0\\
 0\ar[r] &\mathscr{N}_{l/X}\langle \mathcal{T}_{C,x}\rangle \ar[u]\ar[rr] &&\mathscr{N}_{l/X} \ar[u]\ar[rr]
 &&\frac{\mathcal{T}_{X,x}}{\mathcal{T}_{l,x}\oplus \mathcal{T}_{C,x}}
 \ar[u]\ar[r] &0 \\
0\ar[r] &\calO(1)^{r_+-a}\oplus\calO^{a}\ar[u]\ar[rr]
&&\mathrm{Pos}(\mathscr{N}_{l/X}) \ar[u]\ar[rr] &&\C^a \ar[u]\ar[r] &0
}
$$
where all columns and rows are short exact. Note that
$\mathscr{Q}=\calO^2$ if $l$ is of first type and
$\mathscr{Q}=\calO(-1)$ if $l$ is of second type.

If $l$ is of first type and $\mathcal{T}_{C,x}$ is not contained in
$\mathrm{Pos}(\mathscr{N}_{l/X})$, then
$\mathrm{Pos}(\mathscr{N}_{l/X})_x\rightarrow
\frac{\mathcal{T}_{X,x}}{\mathcal{T}_{l,x}\oplus \mathcal{T}_{C,x}}$
is injective and hence $a=n-3$. Since $\mathscr{Q}\cong \calO^2$,
the top row of the above diagram gives
$\mathscr{G}\cong\calO\oplus\calO(-1)$. It follows from the first
column that $\mathscr{N}_{l/X}\langle \mathcal{T}_{C,x}\rangle \cong
\calO^{n-2}\oplus\calO(-1)$ since $r_+=n-3$.

If $l$ is of first type and $\mathcal{T}_{C,x}$ is contained in
$\mathrm{Pos}(\mathscr{N}_{l/X})$, then
$\mathrm{Pos}(\mathscr{N}_{l/X})_x\rightarrow
\frac{\mathcal{T}_{X,x}}{\mathcal{T}_{l,x}\oplus \mathcal{T}_{C,x}}$
has a 1-dimensional kernel and hence $a=n-4$. Since
$\mathscr{Q}\cong \calO^2$, the top row of the above diagram gives
$\mathscr{G}\cong\calO(-1)^2$. It follows from the first column that
$\mathscr{N}_{l/X}\langle \mathcal{T}_{C,x}\rangle
\cong\calO(1)\oplus \calO^{n-4} \oplus \calO(-1)^2$ since $r_+=n-3$.

If $l$ is of second type and $\mathcal{T}_{C,x}$ is not contained in
$\mathrm{Pos}(\mathscr{N}_{l/X})$, then
$\mathrm{Pos}(\mathscr{N}_{l/X})_x\rightarrow
\frac{\mathcal{T}_{X,x}}{\mathcal{T}_{l,x}\oplus \mathcal{T}_{C,x}}$
is injective and hence $a=n-2$. Since $\mathscr{Q}\cong \calO(-1)$,
the top row of the above diagram gives $\mathscr{G}\cong\calO(-1)$.
It follows from the first column that $\mathscr{N}_{l/X}\langle
\mathcal{T}_{C,x}\rangle \cong \calO^{n-2} \oplus \calO(-1)$ since
$r_+=n-2$.

If $l$ is of second type and $\mathcal{T}_{C,x}$ is contained in
$\mathrm{Pos}(\mathscr{N}_{l/X})$, then
$\mathrm{Pos}(\mathscr{N}_{l/X})_x\rightarrow
\frac{\mathcal{T}_{X,x}}{\mathcal{T}_{l,x}\oplus \mathcal{T}_{C,x}}$
has a 1-dimensional kernel and and hence $a=n-3$. Since
$\mathscr{Q}\cong \calO(-1)$, the top row of the above diagram gives
$\mathscr{G}\cong\calO(-2)$. It follows from the first column that
$\mathscr{N}_{l/X}\langle \mathcal{T}_{C,x}\rangle
\cong\calO(1)\oplus \calO^{n-4} \oplus \calO(-2)$ since $r_+=n-2$.
\end{proof}

\begin{cor}\label{cor singular condition}
$S_C$ is singular at $([l],x)$ if and only if $l$ is of second type
and the image of $\mathcal{T}_{C,x}$ in $\mathscr{N}_{l/X,x}$ is
contained in $\mathrm{Pos}(\mathscr{N}_{l/X})_x$.
\end{cor}

\begin{defn}
The curve $C$ has a \textit{bad direction} at a point $x\in C$ if
there exists a line $l$ of second type through $x$ such that the image of
$\mathcal{T}_{C,x}$ in $\mathscr{N}_{l/X,x}$ is contained in
$\mathrm{Pos}(\mathscr{N}_{l/X})_x$. Otherwise, we say that $C$ has
a \textit{good direction} at $x$.
\end{defn}

\begin{lem}\label{lem good direction}
If $C$ is general, then $C$ has good directions everywhere.
\end{lem}
\begin{proof}
Given a line $l\subset X$ of second type and a point $x\in l$, the positive part of $\mathscr{N}_{l/X}$ together with $l$ itself determines an $n-1$ dimensional linear subspace $\mathcal{P}_{l,x}$ of $\mathcal{T}_{X,x}$. Let $F_0\subset F$ be the closed subscheme of lines of second type on $X$. By \cite[Corollary 7.6]{cg}, we know that $\dim F_0\leq n-2$. Let $D_{2}\subset X$ be the locus swept by all lines of second type. Then $\dim D_2\leq n-1$. Let $C$ be a general rational curve on $X$. Then $C$ can meet $D_2$ in at most finitely many points $x_i$, $i=1,\ldots,m$ and through each point $x_i$, there is a unique line $l_i$ of second type. The condition that $C$ has a good condition at $x_i$ is equivalent to $\mathcal{T}_{C,x_i}\notin \mathcal{P}_{l_i,x_i}$. We first note that it is proved in \cite[Lemma 1.4]{izadi} that a general line has good directions everywhere. It is proved in \cite{cs} that the moduli space of degree
$e$ rational curves on $X$ is irreducible. Hence to prove the lemma,
we only need to construct one such curve $C$ whose tangent
directions are all good. This can be obtained by smoothing a chain of lines as follows. Take a chain of lines $L=\cup_{j=1}^e L_j$ such that (i) each line $L_j$ is of first type and has good directions everywhere and (ii) none of the nodes of $L$ is on a line of second type. Since all the components $L_j$ are of first type, the chain $L$ is smoothable; see \cite[Chapter II, Theorem 7.9]{kollar}. Then a general smoothing of $L$ has good directions everywhere since having a bad direction is a closed condition.
\end{proof}

\begin{lem}\label{lem smoothness of Fx}
Let $l\subset X$ be a line of first type and $x\in l$ be a point,
then $F_x$ is smooth at $[l]$.
\end{lem}
\begin{proof}
Deformation theory gives us
$$
\mathcal{T}_{F_x,[l]}=\mathrm{H}^0(l,\mathscr{N}_{l/X}
\otimes\calO_l(-x))=\mathrm{H}^0(l,\calO^{n-3}\oplus\calO(-1)^2)
\cong\C^{n-3}.
$$
Hence $\dim \mathcal{T}_{F_x,[l]}=n-3=\dim F_x$. It follows that
$F_x$ is smooth at $[l]$.
\end{proof}

\begin{proof}[Proof of Lemma \ref{lem smoothness of Sc}]
We first deal with the case when $C$ is a general line. In this
case, we see from \cite[Lemma 1.4]{izadi} that $F_C$ is smooth away
from the point $[C]$ and hence so is $S_C-p^{-1}[C]$. Thus the
singularities of $S_C$ can only appear on $\tilde{C}=p^{-1}[C]\cong
C$. Recall that we have the natural morphism
$q_0=q|_{S_C}:S_C\rightarrow C$ and $\tilde{C}$ is a section of
$q_0$. If $q_0^{-1}y$ is smooth at the point
$\tilde{y}=\tilde{C}\cap q_0^{-1}y$ for some $y\in C$, then $S_C$ is
smooth at $\tilde{y}$. Note that if we identify $q_0^{-1}y$ with
$F_y$, then $\tilde{y}=[C]\in F_y$. It follows that singularities of
$S_C$ can only appear at the points $\tilde{y}\in\tilde{C}$ such
that $F_y$ is singular at the point $[C]$, where
$y=q_0(\tilde{y})\in C$. But this can never happen if $C$ is of
first type by Lemma \ref{lem smoothness of Fx}.

Assume that $e\geq 2$. By Corollary \ref{cor singular condition}, we
only need to show that when $C$ is general, it does not have a bad
direction, which is established in Lemma \ref{lem good direction}.
\end{proof}

\subsection{First order deformations of $C$}

\begin{defn}\label{defn Tv and Sigma_v}
Let $C\subset X$ be a general smooth rational curve of degree $e\geq
2$. Then $C$ is free since it passes a general point of $X$; see \cite[Chapter II, Theorem 3.11]{kollar}. In other words, $\mathscr{N}_{C/X}$ is globally generated. Let $v\in\mathrm{H}^0(C,\mathscr{N}_{C/X})$ be a general section
of the normal bundle. Then $v$ is nowhere vanishing since $\mathscr{N}_{C/X}$ is globally generated with rank $n-1\geq 2$. We define the
rank 2 subbundle $\mathcal{T}_v\subset \mathcal{T}_{X}|_C$,
associated to $v$, to be such that
$$
\mathcal{T}_{v,x}=\{\tau\in\mathcal{T}_{X,x}: \bar{\tau}\in\C
v\subset
\mathscr{N}_{C/X,x}=\frac{\mathcal{T}_{X,x}}{\mathcal{T}_{C,x}}\}
$$
for all $x\in C$. Let
$$
\Sigma_v=\{([l],x)\in S_C: \mathcal{T}_{l,x}\subset \mathcal{T}_{v,x}\}.
$$
\end{defn}

\begin{rmk}\label{rmk Tv and Sigma_v}
Note that we have natural embeddings
$$
\xymatrix{
 S_C\ar@{^(->}[r]^{j_1\quad} &\PP(\mathcal{T}_X|_C)
 &\PP(\mathcal{T}_v)\ar@{_(->}[l]_{\quad j_2},
}
$$
where $j_1$ is defined by $([l],x)\mapsto
\{\mathcal{T}_{l,x}\subset\mathcal{T}_{X,x}\}$ and $j_2$ is induced
by the natural subbundle structure
$\mathcal{T}_v\subset\mathcal{T}_{X}|_C$. We know that $\dim S_C=n-2$, $\dim
\PP(\mathcal{T}_{X}|_C) = n$ and $\dim \PP(\mathcal{T}_v)=2$.
Furthermore, the set $\Sigma_v$ is precisely the intersection of
$\PP(\mathcal{T}_v)$ and $S_C$.
\end{rmk}

\begin{lem}\label{lem Sigma_v}
If $v$ is general, then $S_C$
intersect $\PP(\mathcal{T}_v)$ transversally in the set
$\Sigma_v=\{([L_i],x_i)\}_{i=1}^{r_0}$, where $x_i\neq x_j$ for $i\neq j$. In particular, $\Sigma_v$ is
finite.
\end{lem}
\begin{proof}
By \cite[Lemma 6.5]{cg},
we see that $S_C$ is a fiberwise (2,3) complete intersection of
$\PP(\mathcal{T}_X|_C)$. In other words, for each point $x\in C$, there exist
\[
\phi_{2,x}:\Sym^2(\mathcal{T}){X,x})\rightarrow \C\quad \text{ and }\quad \phi_{3,x}:\Sym^3(\mathcal{T}_{X,x})\rightarrow \C
\]
such that $q_0^{-1}x = F_x\subset \PP(\mathcal{T}_{X,x})$ is defined by $\phi_{2,x}=0$ and $\phi_{3,x}=0$. The tangent direction $\mathcal{T}_{C,x}$ defines a point $\tilde{x}=[\mathcal{T}_{C,x}]\in \PP(\mathcal{T}_{v,x})\subset \PP(\mathcal{T}_{X,x})$. Lemma \ref{lem transversal l and C} implies that the point $\tilde{x}$ is not contained in $F_x$. Note that $\PP(\mathcal{T}_{v,x})$ can be viewed as a line in $\PP(\mathcal{T}_{X,x})$ passing through the point $x$. Since $\mathscr{N}_{C/X}$ is globally generated, we see that for a fixed $x\in C$ the line $\PP(\mathcal{T}_{v,x})$ is a general line passing through $x$. Now since $F_x\subset \PP(\mathcal{T}_{X,x})$ has codimension 2, the line does not meet $F_x$. This shows that when $v$ is general, there are at most finitely many points $x_i\in C$ such that $F_{x_i}$ meets $\PP(\mathcal{T}_{v,x_i})$. For any of such point $x_i$, since $\tilde{x}_i\notin F_x$, the line $\PP(\mathcal{T}_{v,x_i})$ meets $F_{x_i}$ in at most finitely many points. This proves that $S_C$ intersects $\PP(\mathcal{T}_v)$ transversally. If we have $x_i=x_j=x$, then we can make a small deformation of $v$ requiring $\PP(\mathcal{T}_{v,x})\cap F_x\neq\emptyset$. Again since $\mathscr{N}_{C/X}$ is globally generated, a general deformation as above will satisfy that $\PP(\mathcal{T}_{v,x})\cap F_x$ is a single point. This way, we move one of the intersection points above $x$ to a nearby fiber. Hence we have $x_i\neq x_j$ for $i\neq j$.
\end{proof}

\begin{rmk}\label{rmk size of Sigma_v}
The element $v\in\mathrm{H}^0(C,\mathscr{N}_{C/X})$ can be viewed as
a first order deformation of $C$. Note that we have a canonical
identification
$$
\mathcal{T}_{\mathrm{Hilb}(X),[C]}=\mathrm{H}^0(C,\mathscr{N}_{C/X}).
$$
Since $[C]\in\mathrm{Hilb}^e(X)$ is a smooth point, we can find a
smooth pointed curve $\varphi:(T,0)\longrightarrow
(\mathrm{Hilb(X)},[C])$ such that $d\varphi(\mathcal{T}_{T,0})=\C\,
v$. The curve $T$ parameterizes a 1-dimensional family of curves on
$X$,
$$
\xymatrix{
  \mathscr{C}\ar[r]\ar[d] &X\\
  T &
}
$$
such that $\mathscr{C}_0=C$. Let $\Sigma_t=\{L_{i,t}\}$, $t\neq 0$,
be the set of all incidence lines of $\mathscr{C}_t$ and $C$
(meaning lines meeting both $\mathscr{C}_t$ and $C$). We have proved
in \cite[Lemma 3.10]{shen} that the cardinality of $\Sigma_t$ is
equal to $5\deg(\mathscr{C}_t)\deg(C)=5e^2$. If we take the limit as
$t\to 0$ (see \cite[\S11.1]{fulton}), then we see that $\Sigma_t$
specializes to the set $\Sigma_v$. Hence we get
$r_0=|\Sigma_v|=5e^2$. We will write
$$
\Sigma_v=\{([L_i],x_i): i=1,2,\ldots, 5e^2\}.
$$
\end{rmk}

\begin{lem}\label{lem restricted normal bundle L_i}
If $C$ is general, then $L_i$ is of first type and
$$
\mathscr{N}_{L_i/X}\langle \mathcal{T}_{C,x_i}\rangle \cong
\calO^{n-2} \oplus\calO(-1),\quad 1\leq i\leq 5e^2.
$$
Furthermore, we have a decomposition
\begin{equation}\label{eq decomposition of NN}
\mathrm{NN}(\mathscr{N}_{L_i/X}\langle
\mathcal{T}_{C,x_i}\rangle)_{x_i} =
\mathrm{Pos}(\mathscr{N}_{L_i/X})_{x_i} \oplus \mathcal{T}_{C,x_i},
\end{equation}
which is canonical upto a scalar multiplication on
$\mathrm{Pos}(\mathscr{N}_{L_i/X})_{x_i}$.
\end{lem}
\begin{proof}
When $C$ and also $v$ are general, all the lines $L_i$ are of first
type and $\mathcal{T}_{C,x_i}$ is not pointing to the positive part
of $\mathscr{N}_{L_i/X}$. Hence (i) of Proposition \ref{prop cases
of N_0} applies. Consider the following short exact sequence
\begin{equation}\label{eq sequence 1}
\xymatrix{
 0\ar[r] &\mathrm{Pos}(\mathscr{N}_{L_i/X})_{x_i} \ar[r]^{\cdot t_i\quad}
 &\mathrm{NN}(\mathscr{N}_{L_i/X}\langle \mathcal{T}_{C,x_i}\rangle)_{x_i}
 \ar[r] & \overline{\mathcal{T}}_{C,x_i}\ar[r] &0,
 }
\end{equation}
where $t_i$ is a local uniformizer of $C$ at the point $x_i$ and
$\overline{\mathcal{T}}_{C,x_i}$ is the image of
$\mathcal{T}_{C,x_i}$ in $\mathscr{N}_{L_i/X,x_i}$. The natural
homomorphism
$$
\mathcal{T}_{C,x_i}\hookrightarrow \mathcal{T}_{X,x_i}\rightarrow
\mathscr{N}_{L_i/X,x_i}
$$
factors through $\mathrm{NN}(\mathscr{N}_{L_i/X}\langle
\mathcal{T}_{C,x_i}\rangle)_{x_i} \longrightarrow
\mathscr{N}_{L_i/X,x_i}$ and gives a homomorphism
$\theta_i:\overline{\mathcal{T}}_{C,x_i}\rightarrow
\mathscr{N}_{L_i/X}\langle \mathcal{T}_{C,x_i}\rangle)_{x_i}$. This
homomorphism $\theta_i$ canonically splits the sequence \eqref{eq
sequence 1}. After identifying $\overline{\mathcal{T}}_{C,x_i}$ with
${\mathcal{T}}_{C,x_i}$, we get the decomposition. Note that the
only choice made here is the local parameter $t_i$. A different
choice of $t_i$ induces a scalar multiplication on
$\mathrm{Pos}(\mathscr{N}_{L_i/X})_{x_i}$.
\end{proof}

\begin{lem}\label{lem zero dimensional}
Let $V\subset\mathcal{T}_{X,x}$ be a 3-dimensional vector space
containing $\mathcal{T}_{v,x}$, for some $x\in C$. We view $\PP(V)$
as the space of lines tangent to $X$ at $x$ in a direction of $V$.
Then $F_x\cap\PP(V)$ is a zero dimensional scheme of length 6 if
$x\neq x_i$. If $x=x_i$ for some $1\leq i\leq 5e^2$, then
$F_x\cap\PP(V)$ can have at most one component of dimension 1 and
this component has to be a line (on $\PP(V)=\PP^2$) if it exists.
\end{lem}
\begin{proof}
If we show that $F_x\cap\PP(V)$ is zero dimensional, then it of
length $6$ since it is a (2,3)-complete intersection in
$\PP(V)\cong\PP^2$. Assume that $B\subset\PP(V)\cap F_x$ is a curve.
Then $B\cap\PP(\mathcal{T}_{v,x})$ is nonempty, where
$\PP(\mathcal{T}_{v,x})$ is viewed as a line on $\PP(V)$. Or in
other words, $\PP(\mathcal{T}_{v,x})$ contains a line $L$ of $X$.
This exactly means that $(L,x)\in \Sigma_v$. Hence such $B$ does not
exist if $x\neq x_i$. If $x=x_i$, we have already seen that
$\PP(\mathcal{T}_{v,x_i})$ contains the unique line $L_i$ of $X$.
Hence $B\subset \PP(V)$ is of degree 1, namely a line, if it exists. This is because $B$ can only meet the line $\PP(\mathcal{T}_{v,x})$ in the point $[L_i]$.
\end{proof}

\subsection{The rational map $\rho$}

Let $\mathscr{F}$ be the quotient of $\mathcal{T}_{X}|_C$ by
$\mathcal{T}_v$. The natural quotient homomorphism
$\mathcal{T}_{X}|_C\rightarrow \mathscr{F}$ induces a rational map
$\beta:\PP(\mathcal{T}_X|_C)\dashrightarrow \PP(\mathscr{F})$. By
definition, the indeterminacy locus of $\beta$ is exactly
$\PP(\mathcal{T}_v)$. Let $\rho=\beta\circ j_1:S_C\dashrightarrow
\PP(\mathscr{F})$, where $j_1:S_C\hookrightarrow \PP(T_X|_C)$ is the
closed immersion induced by $([l],x)\mapsto[\mathcal{T}_{l,x}\subset
\mathcal{T}_{X,x}]$; see Remark \ref{rmk Tv and Sigma_v}. Then the
indeterminacy locus of $\rho$ is the set $\Sigma_v$. This shows that
we have a morphism
$$
\rho:S_C-\Sigma_v\rightarrow \PP(\mathscr{F}).
$$

\begin{lem}\label{lem pull back by rational map}
Let $f:Y_1\dashrightarrow Y_2$ be a rational map between two smooth
projective varieties whose indeterminacy locus $\Sigma\subset Y_1$ is
a finite set of closed points. Let $\Gamma_f\subset
(Y_1-\Sigma)\times Y_2$ be the graph of $f$ and
$\bar{\Gamma}_f\subset Y_1\times Y_2$ the closure of $\Gamma_f$. We
define the pull-back $f^*$ and push-forward $f_*$ (on either the
cohomology groups or the Chow groups) as
$$
f^*\fa = p_{1,*}(\bar{\Gamma}_f\cdot p_2^*\fa),\qquad f_*\fa =
p_{2,*}(\bar{\Gamma}_f\cdot p_1^*\fa),
$$
where $p_i:Y_1\times Y_2\rightarrow Y_i$ is the projection to the
corresponding factor. Then we have $f^*(\fa^k)=(f^*\fa)^k$, for all
$\fa\in\Pic(Y_2)$ and $k<\dim Y_1$.
\end{lem}
\begin{proof}
Let $E=\Sigma\times Y_2\cap\bar{\Gamma}_f$ be the exceptional set of
the morphism $\bar{\Gamma}_f\rightarrow Y_1$. Let $D$ and $D_0$ be
two divisors on $Y_1\times Y_2$ such that $D\cdot\bar{\Gamma}_f =
D_0\cdot\bar{\Gamma}_f + \gamma_1$ where $\gamma_1$ is a cycle
supported on $E$. We conclude from the above that for all $1\leq
k\leq \dim Y_1$,
\begin{equation}\label{eq induction identity k}
D^k\cdot\bar{\Gamma}_f=D_0^k\cdot\bar{\Gamma}_f+\gamma_k
\end{equation}
where $\gamma_k$ is a cycle supported on $E$. This can be proved by
induction.
In equation \eqref{eq induction identity k}, we
apply $p_{1,*}$ to both sides and note that $p_{1,*}\gamma_k=0$,
$1\leq k\leq\dim Y_1-1$, since $p_1$ contracts $E$ to points while
$\dim\gamma_k \geq 1$. It follows that
$$
p_{1,*}(D^k\cdot\bar{\Gamma}_f) = p_{1,*}(D_0^k\cdot
\bar{\Gamma}_f),\quad 1\leq k\leq \dim Y_1-1.
$$
For any divisor $D$ on $Y_1\times Y_2$, we take
$D_0=p_1^*p_{1,*}(D\cdot\bar{\Gamma}_f)$. Since
$p_1|_{\bar{\Gamma}_f}:\bar{\Gamma}_f\rightarrow Y_1$ is isomorphism
away from $E$, we see that the condition $D\cdot\bar{\Gamma}_f =
D_0\cdot\bar{\Gamma}_f +\gamma_1$ for some cycle $\gamma_1$
supported on $E$. Hence we get
\begin{align*}
p_{1,*}(D^k\cdot\bar{\Gamma}_f) &= p_{1,*}(D_0^k\cdot
\bar{\Gamma}_f) \\
& = p_{1,*}(p_1^*(D')^k\cdot\bar{\Gamma}_f),\qquad D'=p_{1,*}(D\cdot\bar{\Gamma}_f)\\
& = (D')^k\cdot p_{1,*}\bar{\Gamma},\qquad(\text{the projection formula})\\
& = (D')^k =(p_{1,*}(D\cdot\bar{\Gamma}_f))^k
\end{align*}
for all divisors $D$ on $Y_1\times Y_2$ and $1\leq k\leq \dim
Y_1-1$. In particular, if we take $D=p_2^*\fa$ for some
$\fa\in\Pic(Y_2)$, then we get
$$
f^*\fa^k=p_{1,*}(\bar{\Gamma}_f\cdot p_2^*\fa^k)
=p_{1,*}(\bar{\Gamma}_f\cdot(p_2^*\fa)^k)=(p_{1,*}(\bar{\Gamma}_f\cdot
p_2^*\fa))^k=(f^*\fa)^k,
$$
where $1\leq k\leq \dim Y_1-1$. The case $k=0$ is automatic.
\end{proof}

Let $\Gamma_\rho\subset(S_C-\Sigma_v)\times\PP(\mathscr{F})$ be the
graph of $\rho$ and $\bar{\Gamma}_\rho\subset
S_C\times\PP(\mathscr{F})$ its closure. Note that for any
3-dimensional vector space $V\subset \mathcal{T}_{X,x}$ containing
$\mathcal{T}_{v,x}$, the condition
$(([l],x),V/\mathcal{T}_{v,x})\in\Gamma_\rho$ implies that $l\subset
\PP(V)$. We define the closed subvariety $Z\subset
S_C\times\PP(\mathscr{F})$ as
$$
Z =\{(([l],x),V/\mathcal{T}_{v,x})\in S_C\times\PP(\mathscr{F})\mid
l\subset\PP(V)\}.
$$

\begin{lem}\label{lem graph of rho}
$\bar{\Gamma}_\rho = Z$ in
$\mathrm{CH}_{n-2}(S_C\times\PP(\mathscr{F}))$.
\end{lem}
\begin{proof}
From Lemma \ref{lem zero dimensional}, $Z$ agrees with $\Gamma_\rho$
away from $\Sigma_v$, namely
$$
\Gamma_\rho = Z|_{(S_C-\Sigma_v)\times\PP(\mathscr{F})}.
$$
It follows that $\bar{\Gamma}_\rho\subset Z$. Hence we have
$$
Z=\bar{\Gamma}_\rho+ \sum_{i=1}^{5e^2}\{([L_i],x_i)\}\times
\PP(\mathscr{F}_{x_i}),\qquad\text{in
}\mathrm{CH}_{n-2}(S_C\times\PP(\mathscr{F})).
$$
Since $\dim\PP(\mathscr{F}_{x_i})=n-3$, we see that
$\{([L_i],x_i)\}\times \PP(\mathscr{F}_{x_i}) = 0$ in
$\mathrm{CH}_{n-2}(S_C\times\PP(\mathscr{F}))$. Hence the lemma
follows.
\end{proof}

\begin{defn}\label{defn Gamma_v}
Let $\Gamma_v\subset S_C\times S_C$ be the closure of
$(S_C\backslash\Sigma_v)\times _{\PP(\mathscr{F})}
(S_C\backslash\Sigma_v)$.
\end{defn}

\begin{rmk}\label{rmk generic definition Gamma_v}
Let $([l],x)\in S_C$ be a general point. Then $\mathcal{T}_{l,x}$
and $\mathcal{T}_{v,x}$ span a linear 3-dimensional vector space
$V\subset \mathcal{T}_{X,x}$ and
$\rho([l],x)=[V/\mathcal{T}_{v,x}]$. Then by Lemma \ref{lem zero
dimensional}, there are 5 other lines $l_i$, $i=1,\ldots,5$, such
that $\mathcal{T}_{l_i,x}\subset V$. Then $\Gamma_v$ is the
correspondence generically defined by
$$
([l],x)\mapsto \sum_{i=0}^5([l_i],x),
$$
where $l_0=l$.
\end{rmk}

\begin{lem}\label{lem decomposition Gamma_v}
$\Gamma_v = \bar{\Gamma}_\rho^t\circ\bar{\Gamma}_\rho$ in
$\mathrm{CH}_{n-2}(S_C\times S_C)$.
\end{lem}
\begin{proof}
By Lemma \ref{lem graph of rho}, we may replace $\bar{\Gamma}_\rho$
by $Z$. By the definition of composition of correspondences (see
\cite[\S16.1]{fulton}), we know that
\begin{align*}
Z^t\circ Z &= p_{13,*}(p_{12}^*Z\cdot p_{23}^*Z^t)\\
           &= p_{13,*}(Z\times S_C \cap S_C\times Z^t)
\end{align*}
It follows that $Z^t\circ Z$ is represented by the cycle
$$
\{(([l],x),([l'],x))\in S_C\times S_C: [l], [l']\in F_x\cap\PP(V)
\text{ for some }[V/\mathcal{T}_{v,x}]\in\PP(\mathscr{F}_x)\}
$$
One directly checks that this cycle is $\Gamma_v$.
\end{proof}

\begin{lem}\label{Gamma_v lemma}
Let $\,\Gamma_v$ act either on the cohomology groups or on the Chow groups of $S_C$.\\
(i) If $\fa$ is an odd dimensional topological cycle, then
$(\Gamma_v)_* \fa=0$;\\
(ii) If $\fa$ is zero dimensional, then $(\Gamma_v)_* \fa$ is a
constant that only depends on the degree of $\fa$;\\
(iii) Let $\fa=[S_C]$, then $(\Gamma_v)_* \fa=6\fa$;\\
(iv) If $\fa$ is a topological cycle of real codimension $2m$ or an
algebraic cycle of codimension $m$, then $(\Gamma_v)_*\fa$ is a
linear combination of $g^m$ and $g'g^{m-1}$ which only depends on
the intersection numbers $\fa\cdot g^{n-m-2}$ and $\fa\cdot
g'g^{n-m-3}$.
\end{lem}
\begin{proof}
By Lemma \ref{lem decomposition Gamma_v}, we know that
\[
(\Gamma_v)_*=(\bar{\Gamma}_\rho^t\circ\bar{\Gamma}_\rho)_* = (\bar{\Gamma}_\rho^t)_* (\bar{\Gamma}_\rho)_* = \rho^*\rho_*.
\]
Let $\xi$ be the relative
$\calO_{\PP(\mathscr{F})}(1)$ class on $\PP(\mathscr{F})$. Then the
cohomology of $\PP(\mathscr{F})$ is naturally given by
$$
\mathrm{H}^*(\PP(\mathscr{F}))=\mathrm{H}^*(C)[\xi],\quad
\xi^{n-2}+\pi^*c_1(\mathscr{F})\cdot \xi^{n-3}=0.
$$
Where $\pi:\PP(\mathscr{F})\rightarrow C$ is the natural projection.
Similarly, we have the description of the Chow ring as
$$
\mathrm{CH}^*(\PP(\mathscr{F}))=\mathrm{CH}^*(C)[\xi], \quad
\xi^{n-2}+\pi^*c_1(\mathscr{F})\cdot \xi^{n-3}=0.
$$
Consider the following diagram
$$
\xymatrix{
 S_C\ar[r]^{j_1\quad}\ar@{-->}[rd]^\rho &\PP(\mathcal{T}_X|_C)\ar[r]^{\quad\alpha}\ar@{-->}[d]^\beta &C\\
  &\PP(\mathscr{F})\ar[ru]_\pi &
 }
$$
The rational map $\beta$ is defined by the natural homomorphism
$\mathcal{T}_X|_C\rightarrow\mathscr{F}$. The locus where $\beta$ is
not defined is exactly $\PP(\mathcal{T}_v)$. The pull back
$\beta^*\xi$ restricts to a hyperplane class on each fiber of
$\alpha$. The variety $\PP(\mathcal{T}_{X}|_C)$ parameterizes all
lines in $\PP^{n+1}$ which are tangent to $X$ at some point $x\in C$
(remembering the point $x$). Let
$\tilde{g}\subset\PP(\mathcal{T}_{X}|_C)$ be the divisor
corresponding to all the lines meeting a given linear
$\PP^{n-1}\subset\PP^{n+1}$ in general position. Then we see that
$i^*\tilde{g}=g|_{S_C}$ and this is denoted again by $g$ for
simplicity. One also sees that $\tilde{g}$ restricts to a hyperplane
class of $\PP(\mathcal{T}_{X,x})=\alpha^{-1}(x)$. Thus $\tilde{g}$
and $\beta^*\xi$ differ by a class coming from $C$. After pulling
back to $S_C$, we get
$$
\rho^*\xi=g+ rg'
$$
for some integer $r$. By Lemma \ref{lem pull back by rational map},
we get the following key equality
\begin{equation}
\rho^*\xi^k=(\rho^*\xi)^k=(g+rg')^k
\end{equation}
Let $\fa\in \mathrm{H}^{2m+1}(S_C)$ be a topological class of odd
dimension, then $\rho_*\fa$ is an element in
$\mathrm{H}^{2m+1}(\PP(\mathscr{F}))$ which is zero. Hence
$(\Gamma_v)_*(\fa)=\rho^*\rho_*\fa=0$. Now let $\fa$ be an element
of $\mathrm{H}^{2m}(S_C)$ (or $\mathrm{CH}^{m}(S_C)$), $0<m<n-2$
(the cases of $m=0,n-2$ are quite easy to deal with). Then
$\rho_*\fa$ is an element in $\mathrm{H}^{2m}(\PP(\mathscr{F}))$ (or
$\mathrm{CH}^m(\PP(\mathscr{F}))$). Then we have the following
expression
$$
\rho_*\fa=a\xi^{m}+b\pi^*[pt]\cdot\xi^{m-1},
$$
for some $a,b\in\Z$. Apply $\rho^*$ to the above identity and use
Lemma \ref{lem pull back by rational map}, we get
$$
(\Gamma_v)_*(\fa)=\rho^*\rho_*\fa=a(g+rg')^{m}+bg'(g+rg')^{m-1}
=ag^m +(b+ma)g^{m-1}g'
$$
Also, the numbers $a$ and $b$ can be determined in the following way
\begin{align*}
a &= (a\xi^m+b\pi^*[pt]\xi^{m-1})\cdot\pi^*[pt]\xi^{n-m-3}\\
 &=\rho_*\fa\cdot\pi^*[pt]\xi^{n-m-3}\\
 &=\fa\cdot \rho^*(\pi^*[pt]\xi^{n-m-3})\\
 &=\fa\cdot g'(g+rg')^{n-m-3},\qquad (\text{by Lemma \ref{lem pull back by rational
 map}})\\
 &=\fa\cdot g'g^{n-m-3}
\end{align*}
To get $b$, we consider
\begin{align*}
\rho_*\fa\cdot\xi^{n-m-2}
 &=a\xi^{n-2}+b\\
 &=a(-\pi^*c_1(\mathscr{F}))\xi^{n-3}+b\\
 &=-a\deg{\mathscr{F}}+b
\end{align*}
Hence we have
\begin{align*}
b &=a\deg{\mathscr{F}} +\fa\cdot \rho^*\xi^{n-m-2}\\
&=a\deg(\mathscr{F})+\fa\cdot(g+rg')^{n-m-2}
\end{align*}
Thus $b$ only depends on the intersection numbers $\fa\cdot
g^{n-m-2}$ and $\fa\cdot g'g^{n-m-3}$. This finishes the proof.
\end{proof}

\subsection{Proof of Theorem \ref{sigma identity}}

Let $T\subset\mathrm{Hilb}^e(X)$ be a general smooth curve passing
through the point $[C]$. Let $\{\mathscr{C}_t:t\in T\}$ be the
corresponding 1-dimensional family of rational curves on $X$ such
that the fiber $\mathscr{C}_0=C$ at the special point $0=[C]\in T$.
Let $S_t=S_{\mathscr{C}_t}$ and
$i_t=i_{\mathscr{C}_t}:S_t\rightarrow F$. Let $I'_t\subset S_t\times
S_C$, $t\neq 0$, be the natural incidence correspondence, namely
$$
I'_t=\{(([l'],x'),([l],x))\in S_t\times S_C: l\cap
l'\neq\emptyset\}=(i_t\times i_C)^* I\in\mathrm{CH}_{n-2}(S_t\times
S_C).
$$
Let $I_t\subset S_t\times S_C$ be the correspondence generically
defined by
\begin{equation}\label{eq It}
I_t:([l'],x')\mapsto \sum_{i=1}^{5e}([l'_i],x'_i),
\end{equation}
where $([l]',x')\in S_t$ is a general point, $l'_i$ the incidence
lines of $l'$ and $C$, $x'_i=l'_i\cap C$.

\begin{lem}
We have $I_t=I'_t$ in $\mathrm{CH}_{n-2}(S_t\times S_C)$.
\end{lem}
\begin{proof}
Let $\tilde{I}_t=(i_t\times i_C)^{-1}I$. By definition, $I_t$ is
equal to the cycle class of $\tilde{I}_t$. Consider the projection
$p_1:\tilde{I}\rightarrow S_t$. Let $([l'],x')\in S_t$ be any point.
If $l'$ does not meet $C$, then $ p_1^{-1}([l'],x') =
\sum_{i=1}^{5e}([l_i],x_i)$ where $l_i$ run through all the
incidence lines of $l'$ and $C$; if $l'\cap C=y$, then $
p_1^{-1}([l'],x')$ is the union of $F_y=q_0^{-1}y $ and the set of
all secant lines of the pair $(l',C)$. It follows that
$$
\tilde{I}_t= I'_t\cup(\bigcup_{j=1}^{5e^2}
\{[L_{j,t}],x_{j,t})\}\times F_{y_{j,t}}),
$$
where $\{L_{j,t}\}_{j=1}^{5e^2}$ is the set incidence lines of
$\mathscr{C}_t$ and $C$, $x_{j,t}=L_{j,t}\cap\mathscr{C}_t$ and
$y_{j,t}=L_{j,t}\cap C$. Since $\dim F_{y_{j,t}}=n-3$, we see that
the cycle class of $\tilde{I}_t$ is equal to that of $I'_t$.
\end{proof}
Back to the proof of Theorem \ref{sigma identity}. In equation
\eqref{eq It}, when $t$ specializes to 0, $l'$ specializes to a
line $l$ meeting $C$; there are 5 lines, say
$l'_1,\ldots,l'_5$, among all the $l'_i$'s, that specialize to five
lines, $E_i$, $i=1,2,\ldots,5$, passing through $x=l\cap C$; all the
other lines of the $l'_i$'s specialize to secant lines, $l_i$, of the
pair $(l,C)$. We have the following description of the $E_i$'s.

Since $T$ is a smooth curve in the Hilbert scheme of degree $e$
rational curves on $X$, the tangent space $\mathcal{T}_{T,0}$ gives
a one dimensional subspace of
$\mathcal{T}_{\mathrm{Hilb}^e(X),[C]}=\mathrm{H}^0(C,\mathscr{N}_{C/X})$.
Let $v\in \mathrm{H}^0(C,\mathscr{N}_{C/X})$ be a generator of
$\mathcal{T}_{T,0}$. By choosing $T$ general enough, we may assume
that $v$, as a section of $\mathrm{H}^0(C,\mathscr{N}_{C/X})$, is
general. To this $v$, we can associate $\mathcal{T}_v\subset
\mathcal{T}_{X}|_C$ with quotient $\mathscr{F}$, $\Sigma_v\subset
S_C$ and $\rho:S_C\backslash\Sigma_v\rightarrow \PP(\mathscr{F})$ as before.

We use the notation and results of \cite[\S11.1]{fulton}. Let
$I_0=\lim_{t\to 0}I_t$. In the equation \eqref{eq It}, we take the
limit as $t\to 0$, then we get
\begin{equation}\label{eq limit of It}
I_0:([l],x)\mapsto \sum_{i=6}^{5e} ([l'_i],x_i)
+\sum_{i=1}^{5}([E_i],x),
\end{equation}
where $\{l'_i\}_{i=6}^{5e}$ is the set of all secant lines of the
pair $(l,C)$ and $x_i=l'_i\cap C$. Note that the rule in \eqref{eq
limit of It} generically defines $I_C+\Gamma_v-\Delta_{S_C}$. We see
that the difference of $I_0$ and $I_C+\Gamma_v-\Delta_{S_C}$ is
supported on $\Sigma_v\times_C\PP(\mathscr{F})$. Since $\dim
\Sigma_v\times_C\PP(\mathscr{F})=n-3$, we have
\begin{equation}\label{identity on correspondence}
I_0=I_C+\Gamma_v-\Delta_{S_C},\quad \text{ in
}\mathrm{CH}_{n-2}(S_C\times S_C).
\end{equation}
Note that $I_t=I|_{S_t\times S_C}$ for $t\neq 0$. By taking limits
and applying \cite[Proposition 11.1]{fulton}, we know that the class
of $I_0$ is equal to $I$ restricted to $S_C\times S_C$.
Equivalently, we have $ I_0=I'_C$ in $\mathrm{CH}_{n-2}(S_C\times
S_C)$. Thus we get the following key identity
\begin{equation}\label{eq key identity Ic and I'c}
I'_C=I_C +\Gamma_v -\Delta_{S_C}.
\end{equation}
Since $\Psi_C=\Psi\circ(i_C)_*$, $\Phi_C=(i_C)^*\circ\Phi$ and
$\Phi\circ\Psi=I_*$, we see that
$$
\Phi_C\circ\Psi_C=(i_C)^*\circ I_*\circ (i_C)_*=((i_C\times
i_C)^*I)_*=(I'_C)_*
$$
Combine this with the equation \eqref{eq key identity Ic and I'c},
we obtain
$$
\Phi_C\circ\Psi_C = \sigma -1 +(\Gamma_v)_*
$$
as actions on either the cohomology groups or the Chow groups of
$S_C$. Then Theorem \ref{sigma identity} follows easily from Lemma
\ref{Gamma_v lemma}.

\subsection{The subalgebra $\Q[g,g']$}

On $S_C$ we have the polarization $g$ and the class $g'=q_0^*[pt]$,
where $q_0=q|_{S_C}:S_C\rightarrow C$ is the natural morphism.

\begin{lem}\label{multiple of h lemma}
$\Psi_C(g'\cdot\Phi_C(h^{m}))=2h^{m+1}$.
\end{lem}

\begin{proof}
First we note that
$\Psi_C(g'\cdot\Phi_C(h^m))=\Psi(F_x\cdot\Phi(h^m))$. Let $M\subset
X$ be a general complete intersection of hyperplanes that represents
$h^m$. Choose $x\in X$ to be a general point. Similar to the proof
of Theorem \ref{fundamental relation}, we consider the following
diagram
$$
\xymatrix{
 D_1+D_2+D \ar[r]\ar[d] &Y\ar[r]\ar[d] &X\ar[d]\\
 P\ar[r]\ar[d]^\pi &\mathrm{G}(1,2,n+2)\ar[r]\ar[d] &\PP^{n+1}\\
 M\ar[r]^\varphi &\mathrm{G}(2,n+2) &
}
$$
where all the squares are fiber products; $\varphi(x')$ is the line
through $x$ and $x'$, for all $x'\in M$; $D_1\subset P$ is a section
of $\pi$ and the natural morphism $g_1:D_1\rightarrow X$ contracts
$D_1$ to $x\in X$; $D_2$ is a section of $\pi$ and the natural
morphism $g_2:D_2\rightarrow X$ maps $D_2$ isomorphically to
$M\subset X$; $D$ corresponds to the third point in the intersection
of the lines with $X$. Let $g:D\rightarrow X$ be the natural
morphism. Let $\Delta\subset M$ be the intersection of $M$ and $D_x$
(recall that $D_x$ is the variety swept out by lines through $x$).
Then, by Bertini's theorem, $\Delta$ is smooth of codimension 2 in
$M$. As in Lemma \ref{lem blow up}, the natural map
$\pi':D\rightarrow M$ is the blow-up along $\Delta$. Let $E\subset
D$ be the exceptional divisor. As before, we have
\begin{equation}\label{eq id h E}
h|_D+E=2(\pi|_D)^*(h|_M)
\end{equation}
where $h|_D=g^*h$. One checks that the push-forward of
$\pi^*(h|_M)|_{(D_1+D_2+D)}$ to $X$ is equal to $3h^{m+1}$ and that
\begin{align*}
 (g_2)_*(\pi^*(h|_M)|_{ D_2} )& = M\cdot h = h^{m+1}\\
 (g_1)_*(\pi^*(h|_M)|_{ D_1} )& = 0
\end{align*}
It follows that
$$
g_*(\pi^*(h|_M)|_D) = 3h^{m+1}-h^{m+1}=2h^{m+1}.
$$
Note that the push-forward of $D+D_1+D_2$ to $X$ is $P|_X=3h^m$. We
also know that $(g_1)_*D_1=0$ and $(g_2)_*D_2=M=h^{m}$. Hence we get
$g_*D=2h^m$. As a consequence, we have $g_*(h|_D)=g_*D\cdot
h=2h^{m+1}$. We apply $g_*$ to the equality \eqref{eq id h E} and
note that $g_*E=\Psi(F_x\cdot\Phi(h^m))$. Then we get
$$
\Psi(F_x\cdot \Phi(h^m)) =
2g_*(\pi'^*h|_M)-g_*(g^*h)=2(2h^{m+1})-2h^{m+1}=2h^{m+1}.
$$
Hence the lemma follows.
\end{proof}

\begin{prop}\label{preserving prop}
The following statements hold in $\mathrm{H}^*(S_C,\Q)$.

(1) $g^m$ is a linear combination of $\Phi_C(h^{m+1})$ and
$g'\Phi_C(h^m)$; $g'g^{m-1}$ is a multiple of $g'\Phi_C(h^m)$. The
algebra $\Q[g,g']$ is generated by cycles of the form
$\Phi_C(h^{m+1})$ with $0\leq m\leq n-2$ and $g'\Phi_C(h^{m})$ with
$1\leq m \leq n-2$.

(2) The restriction of $\mathrm{H}^*(G)$ to $\mathrm{H}^*(S_C)$ is
equal to $\Q[g,g'g]\subset\Q[g,g']$.

(3) The action of $\sigma$ preserves the algebra $\Q[g,g']$.

(4) Under $\Psi_C$, the image of $\Q[g,g']$ is $\Q[h]_{0<\deg<n}$
where $h$ is the class of a hyperplane on $X$.
\end{prop}

\begin{proof}
Before proving the proposition, we review a bit of Schubert calculus.
We refer to \S14.7 of \cite{fulton} for the details of the general theory.
Schubert calculus on $G=\mathrm{G}(2,n+2)$ shows that
$\mathrm{CH}^m(G)=\mathrm{H}^{2m}(G)$ is generated by Schubert
varieties that are defined by flags $\PP^a\subset\PP^b\subset
\PP^{n+1}$ with $a< b\leq n+1$ and $a+b=2n+1-m$. A Schubert variety
is the space of lines meeting $\PP^a$ and contained in $\PP^b$. In our case,
most of these classes restrict to 0 on $S_C$
unless $b=n$ or $b=n+1$. If $b=n+1$ and $a=n-m$, where $0\leq m\leq
n-1$, then the corresponding Schubert class restricts to
$\Phi(h^{m+1})$ on $F$ and hence restricts to $\Phi_C(h^{m+1})$ on
$S_C$. If $b=n$ and $a=n-m+1$,where $2\leq m\leq  n$, then the
corresponding Schubert class restricts to
$\deg(C)\,g'\cdot\Phi_C(h^m)$ on $S_C$.
Here we note that $F_y\subset S_C$ represents $g'$ on $S_C$, for any
$y\in C$. As a cycle on $F$, the class of $F_y$ comes from the
restriction of some class from $G$. However, viewed as a class on
$S_C$, $F_y$ is not the restriction of any class from $G$. The
degree of a nonzero homogeneous element of $\Q[g,g']$ means the
degree of that element as a polynomial.

Since $g^m$ comes from $G$, we know that it can be written as a
linear combination of $\Phi_C(h^{m+1})$ and $g'\Phi_C(h^m)$. Hence
we also see that $g'g^{m-1}$ is a multiple of $g'\Phi_C(h^m)$ since
$(g')^2=0$. This proves (1).

The discussion of Schubert calculus above shows that the restriction
of $\mathrm{H}^*(G)$ to $S_C$ is generated by $\Phi_C(h^{m+1})$,
$0\leq m\leq n-2$ and $g'\Phi_C(h^m)$, $2\leq m\leq n-2$. Combine
this fact with (1), we see that the the image of the restriction
$\mathrm{H}^*(G)\rightarrow \mathrm{H}^*(S_C)$ consists of elements
$f\in\Q[g,g']$ such that the coefficient of the term $g'$ in $f$ is
zero. This proves that the image
of $\mathrm{H}^*(G)\rightarrow \mathrm{H}^{*}(S_C)$ is $\Q[g,g'g]$.

To prove (3), we only need to consider the action of $\sigma$ on
$\Phi_C(h^{m+1})$ and $g'\Phi_C(h^m)$. By Proposition \ref{sigma
identity}, the following equations hold modulo $\Q[g,g']$.
\begin{align*}
\sigma(\Phi_C(h^{m+1})) &= \Phi_C(h^{m+1})+
\Phi_C\Psi_C(\Phi_C(h^{m+1}))\\
&=\Phi_C\circ\Psi(\Phi(h^{m+1})\cdot{F_C})\\
&=\begin{cases}
 \Phi_C(3e\deg(h^{m+1})h^{m+1} -2eh^{m+1}), & m<n-2,\\
 \Phi_C(3e\deg(h^{m+1})h^{m+1} -2eh^{m+1} -2\deg(h^{n-1})C), &m=n-2
\end{cases}\\
&=\begin{cases}
 0, &m<n-2\\
 -6\Phi_C(C), &m=n-2
\end{cases}
\end{align*}
Here the first equation uses Proposition \ref{sigma identity}; the
second equation uses the fact that $\Phi_C(h^{m+1})\in\Q[g,g']$; the
third equation uses Theorem \ref{fundamental relation}; the last one
uses the fact that the degree of $h^{m}$ with respect to the
polarization $h$ is 3. In the case $m=n-2$, we have that the
cohomology class of $\Phi_C(C)$ is
$5e^2[pt]\in\mathrm{H}^{2n-4}(S_C,\Q)$. Note that
$\mathrm{H}^{2n-4}(S_C,\Q)=\Q[pt]=\Q[g,g']_{\deg=n-2}$. Hence we
also have $\Phi_C(C)\in\Q[g,g']$. Similarly, the following
equalities hold modulo $\Q[g,g']$.
\begin{align*}
\sigma(g'\cdot \Phi_C(h^m)) &= \Phi_C\circ\Psi_C(g'\cdot
\Phi_C(h^m))
+g'\cdot\Phi_C(h^m)\\
& = 2\Phi_C(h^{m+1})+g'\Phi_C(h^m)
\end{align*}
where we use Lemma \ref{multiple of h lemma}.

To prove (4), we also consider the action of $\Psi_C$ on
$\Phi_C(h^{m+1})$ and $g'\Phi_C(h^m)$. By Theorem \ref{fundamental
relation} and Lemma \ref{multiple of h lemma}, we know that the
images are multiples of $h^{m+1}$, $0\leq m\leq n-2$ (we have this
range since otherwise the cycles on $S_C$ are automatically zero for
dimension reasons).
\end{proof}

\section{The quadratic relation and Prym-Tjurin construction}

Let $C\subset X$ be a general smooth rational curve on $X$ with
$e=\deg(C)\geq 2$. Let $S_C=q^{-1}(C)$. As was shown in Lemma \ref{lem smoothness of Sc}, $S_C$ is smooth and it is the normalization of the variety
of all lines meeting $C$. We will use the notation from the previous
section.

\begin{defn}
We define the \textit{primitive cohomology},
$\mathrm{H}^*(S_C,\Z)^\circ$, of $S_C$, to be
$$
\{\alpha\in \mathrm{H}^*(S_C,\Z):\alpha\cup
\beta=0,\,\forall\beta\in\Q[g,g']\}.
$$
We define the \textit{primitive Chow group},
$\mathrm{CH}^*(S_C)^\circ$, of $S_C$ , to be
$$
\{\alpha \in\mathrm{CH}^*(S_C):\alpha\cdot
\beta=_{num}0,\,\forall\beta\in\Q[g,g']\}.
$$
An element $\fa\in \mathrm{H}^*(S_C)^\circ$ is called a
\textit{primitive cohomology class}; an element
$\fa\in\mathrm{CH}^*(S_C)^\circ$ is called a \textit{primitive cycle
class}.
\end{defn}

\begin{rmk}
Since the restriction of $\mathrm{H}^*(G,\Q)$ to $S_C$ is
$\Q[g,gg']$, we deduce that $H^{n-2}(S_C,\Z)^\circ$ ($n\neq 4$)
consists of elements $\alpha$ such that $\alpha\cup\beta=0$ for all
$\beta$ coming from $G$. Note that the only difference between $\Q[g,g']$
and $\Q[g,gg']$ is the element $g'$ since $(g')^2=0$; when $n\neq 4$, the equation $\alpha\cup g'=0$
is automatic for dimension reasons since $g'\in \mathrm{H}^2(S_C)$ and
$\alpha\in\mathrm{H}^{n-2}(S_C)$ are not of complementary dimension.
If $n$ is odd, then every cohomology class in $\mathrm{H}^{n-2}(S_C)$
is primitive by definition.
\end{rmk}

By Proposition \ref{preserving prop}, the action $\sigma$ induces an
action, still denoted by $\sigma$, on the primitive
cohomology and the primitive Chow groups of $S_C$. This is because
$\sigma$ is symmetric and preserves the classes in $\Q[g,g']$. Hence
$\alpha\cdot\sigma(\beta)=\sigma(\alpha)\cdot\beta$. If $\alpha$ is
a primitive class and $\beta\in\Q[g,g']$, then the above identity
shows that $\sigma(\alpha)\cdot\beta=0$ and hence $\sigma(\alpha)$
is still primitive.

\begin{thm}\label{main theorem}
Let $C\subset X$ be a general rational curve of degree $e\geq 2$ as
above. Let $\sigma$ be the action of the incidence correspondence on
either $\mathrm{H}^*(S_C)^\circ$ or $\mathrm{CH}^*(S_C)^\circ$. Then
the following are true.

(1) On the primitive part of either the cohomology groups or the
Chow groups, $\sigma$ satisfies the following quadratic relation
$$
(\sigma-1)(\sigma +2e-1)=0
$$

(2) The map $\Phi_C$ induces an isomorphism of Hodge structures
$$
\Phi_C:\mathrm{H}^{n}(X,\Z)_{\text{prim}} \rightarrow
P(\mathrm{H}^{n-2}(S_C,\Z)^\circ,\sigma)(-1)
$$
The intersection forms are related by the following identity
$$
\Phi_C(\alpha)\cdot\Phi_C(\beta)=-2e\,\alpha\cdot\beta
$$

(3) The map $\Phi_C$ induces an isomorphism
$$
\Phi_C:\mathrm{A}_i(X)_\Q\rightarrow
P(\mathrm{CH}_{i-1}(S_C)^\circ_\Q,\sigma)
$$
\end{thm}

\begin{proof}
Since $\Gamma_v$ is also symmetric, we know that it acts on the
primitive cohomology and the primitive Chow groups. But by Lemma
\ref{Gamma_v lemma}, the image of $\Gamma_v$ is always non-primitive
unless it is zero. Hence we get that $(\Gamma_v)_*=0$ on primitive
cohomology and Chow groups. Hence on the primitive part of the
cohomology group and the Chow groups, we have
$$
\sigma=\Phi_C\circ\Psi_C+1.
$$
The next fact that we need is
\begin{lem}
If $\fa$ is a primitive class in either the cohomology group or the
Chow groups of $S_C$, then $\Psi_C(\fa)$ has $h$-degree zero.
\end{lem}
The proof of the above lemma is easy. We note that $\Psi_C(\fa)\cdot
h^i=\fa\cdot\Phi_C(h^i)=0$, since $\Phi_C(h^i)\in\Q[g,g']$; see (1)
of Proposition \ref{preserving prop}. Now we can prove statement (1)
of the theorem. Theorem \ref{fundamental relation} shows that, on
the primitive cohomology and the primitive Chow groups, we have
$$
\Psi_C(\Phi_C(\fa)) + 2e\fa =0.
$$
From this we get
\begin{align*}
(\sigma-1)(\sigma+2e-1)(\fa)& = \Phi_C\circ\Psi_C(\Phi_C\circ\Psi_C
+2e)(\fa)\\
&=\Phi_C(\Psi_C\circ\Phi_C(\Psi_C(\fa)))+2e\Phi_C\circ\Psi_C(\fa)\\
&=\Phi_C(-2e\Psi_C(\fa)) + 2e\Phi_C\circ\Psi_C(\fa),\quad (\text{by
Theorem \ref{fundamental relation}})\\
&=0
\end{align*}
Now we prove (2). For simplicity, we write $P$ for
$P(\mathrm{H}^{n-2}(S_C,\Z)^\circ,\sigma)$. Since
$\Phi_C\circ\Psi_C=\sigma-1$ and that $\Psi_C$ is onto (Propositioon
\ref{surjectivity of Psi}), we know that the image of $\Phi_C$ is
exactly $P$. By Theorem \ref{fundamental relation}, we know that
$\Psi_C\circ\Phi_C=-2e$. This implies that $\Phi_C$ is injective.
Hence $\Phi_C:H^n(X,\Z)_{prim}\rightarrow P$ is an isomorphism.
$\Phi_C$ respects the Hodge structures and hence is an isomorphism
of Hodge structures. The intersection forms are related by
\[
 \Phi_C(\alpha)\cdot \Phi_C(\beta) = \alpha\cdot \Psi_C(\Phi_C(\beta))
= \alpha\cdot (-2e\beta) = -2e \alpha\cdot \beta
\]

Statement (3) can be proved exactly in the same
way.
\end{proof}

\begin{prop}\label{surjectivity of Psi}
(1) The homomorphism
$$
\Psi_C:\mathrm{H}^{n-2}(S_C,\Z)^\circ\rightarrow
\mathrm{H}^n(X,\Z)_{prim}
$$
on primitive cohomology is surjective. \\
(2) The homomorphism
$$
\Psi_C:\mathrm{CH}_m(S_C,\Q)^\circ\rightarrow \mathrm{A}_{m+1}(X,\Q)
$$
on primitive Chow groups is surjective.
\end{prop}

\begin{proof}
Statement (1) follows from Theorem \ref{surjectivity on primitive} (for $n\neq 4$) and Remark \ref{rmk n equals 4} (for $n=4$).
The proof of (2) is easy since we have $\Q$ coefficients. Let
$\alpha\in \mathrm{CH}_{m+1}(X)_\Q$, take $\fa=-\frac{1}{2e}\Phi_C(\alpha)$.
Then we have $\alpha=\Psi_C(\fa)$. Now assume that $\alpha$ has
$h$-degree 0. For any $\fb\in\Q[g,g']$, we have
$\Phi_C(\alpha)\cdot\fb=\alpha\cdot\Psi_C(\fb)=0$ since
$\Psi_C(\fb)\in\Q[h]$. This means that $\Phi_C(\alpha)$ is a
primitive cycle class. This gives the surjectivity of $\Psi_C$ on
primitive Chow groups.
\end{proof}

Let $\pi:\mathscr{C}\rightarrow T$ be a family of curves together
with a morphism $f:\mathscr{C}\rightarrow X$. Then we can define the
Abel-Jacobi homomorphism $\Phi_T:=\pi_*f^*$ and the cylinder
homomorphism $\Psi_T:=f_*\pi^*$ on the cohomology groups and
the Chow groups. Let $C\subset X$ be a general smooth rational curve
of degree $e\geq 2$ as before. Then there is a natural incidence
correspondence $\Gamma_{T,C}\subset T\times S_C$ given by
$$
t\mapsto \sum([l_i],x_i)
$$
where $l_i$ are the incidence lines of $\mathscr{C}_t$ and $C$ (i.e.
lines meeting both curves) and $x_i=l_i\cap C$. This correspondence
induces
$$
(\Gamma_{T,C})_*:\mathrm{CH}_r(T)\rightarrow \mathrm{CH}_r(S_C)
$$
and
$$
[\Gamma_{T,C}]^*:\mathrm{H}^{n-2}(S_C,\Z)\rightarrow\mathrm{H}^{n-2}(T,\Z)
$$
We can define the primitive cohomology groups and the primitive Chow
groups of $T$ as the classes that are orthogonal to $\Phi_T(\Z[h])$.
Hence $\Gamma_{T,C}$ also induces homomorphism between the primitive
cohomology groups and the primitive Chow groups.

\begin{prop}
Let $\mathscr{C}\rightarrow T$ be a family of curves on $X$ and
$C\subset X$ a general rational curve of degree $e$ as above. All
the homomorphisms are restricted to primitive classes. Then the
following holds.

(1) The homomorphism $[\Gamma_{T,C}]^*:\mathrm{H}^{n-2}(S_C,\Z)^\circ \rightarrow \mathrm{H}^{n-2}(T,\Z)$ factors as $\Phi'_T\circ(\sigma -1)$, where $\Phi'_T:P(\mathrm{H}^{n-2}(S_C,\Z)^\circ,\sigma) \rightarrow \mathrm{H}^{n-2}(T,\Z)$ is a homomorphism such that the following diagram is commutative.
\[
\xymatrix{
 \mathrm{H}^{n-2}(S_C,\Z)\ar[d]_{\sigma -1}\ar[r]^{\Psi_C} &\mathrm{H}^{n}(X,\Z)_{prim}\ar[d]^{\Phi_T}\ar[ld]_{\Phi_C}\\
 P(\mathrm{H}^{n-2}(S_C,\Z)^\circ,\sigma)\ar[r] ^{\qquad\Phi'_T} &\mathrm{H}^{n-2}(T,\Z)
}
\]

(2) This image of $(\Gamma_{T,C})_*:\mathrm{CH}_r(T,\Q)^\circ\rightarrow \mathrm{CH}_r(S_C,\Q)^\circ$ is contained in the subgroup
$P(\mathrm{CH}_r(S_C,\Q)^\circ,\sigma)$, where $\mathrm{CH}_r(T,\Q)^\circ$ is the subgroup of primitive elements. In other words, the following diagram is commutative.
\[
\xymatrix{
  &\mathrm{A}_{r+1}(X,\Q)\ar[rd]^{\Phi_C} &\\
  \mathrm{CH}_r(T,\Q)^\circ \ar[rr]^{(\Gamma_{T,C})_*}\ar[ru]^{\Psi_T} &&P(\mathrm{CH}_r(S_C,\Q)^\circ,\sigma)
}
\]
\end{prop}

\begin{proof}
These are consequences of the identities
$[\Gamma_{T,C}]^*=\Phi_T\circ\Psi_C$ and
$(\Gamma_{T,C})_*=\Phi_C\circ\Psi_T$.
\end{proof}

Let $C_1$ and $C_2$ be two general rational curves on $X$ of degree
$e_1$ and $e_2$ respectively. Then there is a natural incidence
correspondence $\Gamma_{12}\subset S_{C_1}\times S_{C_2}$ defined by
$$
([l],x)\mapsto \sum_{i=1}^{5e_2}([l_i],x_i)
$$
where $([l],x)\in S_{C_1}$, $l_i$ are the incidence lines of $l$ and
$C_2)$, $x_i=l_i\cap C_2$. Let $\gamma_{12}=(\Gamma_{12})_*$ be the
induced homomorphism on either the cohomology groups or the Chow
groups. Note that by definition, we have
$\gamma_{12}=\Phi_{C_2}\circ\Psi_{C_1}$. Since both $\Phi_C$ and
$\Psi_C$ respects primitive classes, the above identity implies that
$\gamma_{12}$ takes primitive classes to primitive classes. We still
use $\gamma_{12}$ to denote the action on the primitive cohomology
and the primitive Chow groups.

\begin{prop}
Let $C_1,\,C_2\subset X$ be two general rational curves of degree at
least 2 and $\gamma_{12}$ be the homomorphism induced by incidence
correspondence as above. Let $\sigma_1$ and $\sigma_2$ be the action
of the self incidence correspondence on $S_{C_1}$ and $S_{C_2}$ respectively.
Then the following are true.

(1) The image of
$\gamma_{12}:\mathrm{H}^{n-2}(S_{C_1},\Z)^\circ\rightarrow
\mathrm{H}^{n-2}(S_{C_2},\Z)^\circ$ is always in the Prym-Tjurin
part. Furthermore there is an isomorphism of Hodge structures
$$t_{12}:P(\mathrm{H}^{n-2}(S_{C_1},\Z)^\circ,\sigma_1)\rightarrow
P(\mathrm{H}^{n-2}(S_{C_2},\Z)^\circ,\sigma_2)$$ such that
$\Phi_{C_2}=t_{12}\circ\Phi_{C_1}$ and
$\gamma_{12}=t_{12}(\sigma_1-1)$.

(2) The same conclusions as in (1) hold for primitive Chow groups
with $\Q$-coefficient.
\end{prop}
\begin{proof}
For simplicity, we write
$\Lambda_i=\mathrm{H}^{n-2}(S_{C_i},\Z)^\circ$ for the primitive
cohomology and $P_i=\mathrm{Im}(\sigma_i-1)$ for the Prym-Tjurin
part, $i=1,2$. Let $\Lambda=\mathrm{H}^n(X,\Z)_{prim}$,
$\Phi_i=\Phi_{C_i}:\Lambda\rightarrow P_i$. Then one easily checks
that $t_{12}=\Phi_2\circ\Phi_1^{-1}$ satisfies (1). The proof of (2)
goes in the same way.
\end{proof}

\section{Surjectivity of $\Psi_C$ on primitive cohomology}
In this section we supply a proof of the surjectivity of $\Psi_C$ on
the primitive cohomologies. To do this, it is more convenient to
consider homology instead of cohomology. We define
\begin{equation}
V_{n-2}(S_C,\Z)=\ker\{\mathrm{H}_{n-2}(S_C,\Z)\rightarrow\mathrm{H}
_{n-2}(G,\Z)\},
\end{equation}
and
\begin{equation}
V_n(X,\Z)=\ker\{\mathrm{H}_n(X,\Z)\rightarrow
\mathrm{H}_n(\PP^{n+1},\Z)\}.
\end{equation}
Under the Poincar\'e duality $\mathrm{H}_{n-2}(S_C,\Z)\cong
\mathrm{H}^{n-2}(S_C,\Z)$, the subspace $V_{n-2}(S_C,\Z)$
corresponds to $\mathrm{H}^{n-2}(S_C,\Z)^\circ$ if $n\neq 4$. When
$n=4$, since the class $g'$ is not from $G(2,6)$, we see that
$V_2(S,\Z)$ is bigger than $\mathrm{H}^2(S_C,\Z)^\circ$. Actually,
we have
$$
\mathrm{H}^2(S_C,\Z)^\circ = \{\sigma\in V_2(S_C,\Z): g'\cdot
\sigma=0\}.
$$
Similarly the Poincar\'e duality on $X$ allows us to identify
$V_n(X,\Z)$ with $\mathrm{H}^n(X,\Z)_{prim}$. Then the surjectivity
of
$$
\Psi_C:\mathrm{H}^{n-2}(S_C,\Z)^\circ\rightarrow
\mathrm{H}^n(X,\Z)_{prim}
$$
when $n\neq 4$, is equivalent to the following

\begin{thm}\label{surjectivity on primitive}
The natural cylinder homomorphism
$$
\Psi_C: V_{n-2}(S_C,\Z)\rightarrow V_n(X,\Z)
$$
is surjective.
\end{thm}
The idea of the proof is the Clemens-Letizia method, see
\cite{clemens} and \cite{letizia}. Our presentation closely follows
that of \cite[\S2,3]{shimada}. Let $\pi:V\rightarrow\Delta$ be a proper
flat holomorphic map from a complex manifold $V$ of dimension $m+1$
onto the unit disk $\Delta$. This map is called a
\textit{degeneration} if $\pi$ is smooth over the punctured disk
$\Delta^*=\Delta-0$ and $V_t:=\pi^{-1}(t)$ is irreducible for $t\neq
0$. Let $\mathrm{Sing}(V_0)$ denote the singular locus of $V_0$.

\begin{defn}(\cite[Definition 1]{shimada})
A degeneration $\pi:V\rightarrow \Delta$ is called \textit{quadratic
of codimension} $r$ if $\mathrm{Sing}(V_0)$ is connected and, for
every point $p\in\mathrm{Sing}(V_0)$, there exist local coordinates
$(z_0,\ldots,z_m)$ of $V$ around $p$ such that
$\pi=z_0^2+\cdots+z_r^2$.
\end{defn}

\begin{prop}\label{Lefschetz pencil}
For any smooth cubic hypersurface $X\subset\PP^{n+1}$, there exist a
Lefschetz pencil $X_t\subset\PP^{n+1}$ with $t\in B\cong\PP^1$, such
that

(1) The total space, $\mathcal{F} = \coprod_{t\in B} F(X_t)$, of
the associated Fano schemes of lines is smooth.

(2) $X_0=X$ is the cubic hypersurface we start with and
$X_t$ is smooth unless $t\in\{t_1,\ldots,t_N\}$.

(3) For each degeneration point $t_j\in B$,
$j=1,\ldots,N$, the family $\beta:\mathcal{F}\rightarrow B$ is a
quadratic degeneration of codimension $n-2$ at the point $t_j$.
\end{prop}
This Proposition can be viewed as a special case of
\cite[Proposition 1]{shimada}. Note that we have no isolated
singularities in $F(X_t)$ since $F(X_t)$ is smooth as long as $X_t$
is smooth. By the results of \cite[Theorem 7.8]{cg}, we easily get a
description of the fibers of $\beta$. If $t\in B-\{t_1,\ldots,t_N\}$
then $\mathcal{F}_t=F(X_t)$ is smooth of dimension $2n-4$. Let $x_i$
be the ordinary double point of $X_{t_i}$ and $\Gamma_i$ be the
lines on $X_{t_i}$ that pass through $x_i$. We know that $\Gamma_i$
is smooth of dimension $n-2$. The singular fiber $\mathcal{F}_{t_i}$
is irreducible with $\Gamma_i$ being the singular locus which is an
ordinary double subvariety. Let $z\in\Gamma_i$ be a point and
$Y\subset \mathcal{F}$ be an $(n-1)$-dimensional complex submanifold
is a small neighborhood of $z$. If $Y$ meets $\Gamma_i$
transversally in the point $z$, then $z$ is a non-degenerate
critical point of $\beta|_Y$. There is an associated vanishing cycle
$\sigma_i\in\mathrm{H}_{n-2}(Y_{t_i+\epsilon},\Z)$.

The next observation we make is that a cubic hypersurface with an
isolated ordinary double point has enough rational curves in its
smooth locus. Let $X\subset\PP^{n+1}$ be a cubic hypersurface with
an ordinary double point $x_0\in X$. Then the projection from the
point $x_0$ defines a morphism $\tau:\tilde{X}\rightarrow \PP^n$,
where $\tilde{X}=\mathrm{Bl}_{x_0}(X)$ is the blow-up of $X$ at the
point $x_0$. Let $\sigma: \tilde{X}\rightarrow X$ be the blow-up
morphism. Then it is known that $\tau$ is the blow-up of $\PP^n$
along a smooth $(2,3)$-complete intersection $Z\subset\PP^n$; see
\cite[Lemma 6.5]{cg}. Let $Q\subset \PP^n$ be the unique quadric
hypersurface containing $Z$ and $\tilde{Q}\subset\tilde{X}$ be the
strict transform of $Q$. Then $\sigma$ is an isomorphism on
$\tilde{X}\backslash\tilde{Q}$ and it contracts $\tilde{Q}$ to the
singular point $x_0$. Let $h\in\Pic(X)$ be the class of a hyperplane
section and $H\in\Pic(\PP^n)$ the class of $\calO(1)$. Then on
$\tilde{X}$, we have
$$
\sigma^*h = 3\tau^*H-E.
$$
Let $C\subset X\backslash \{x_0\}$ be a degree $e$ rational curve
and $\tilde{C}=\sigma^{-1}C\subset\tilde{X}$. Then this curve
$\tilde{C}$ satisfies the following conditions
$$
 \tilde{C}\cdot (3\tau^* H-E)=e,\qquad \tilde{C}\cdot (2\tau^*
 H-E)=\tilde{C}\cdot\tilde{Q}=0.
$$
This implies that $\tilde{C}\cdot\tau^*H=e$ and $\tilde{C}\cdot
E=2e$. Hence $C'=\tau(\tilde{C})\subset \PP^n$ is a degree $e$
rational curve which meets $Z$ in $2e$ points and $\tilde{C}$ is the
strict transform of $C'$. Conversely, if we start with a degree $e$
rational curve $C'\subset\PP^n$ which meets $Z$ in $2e$ points. Let
$\tilde{C}$ be the strict transform of $C'$. Then
$C=\sigma(\tilde{C})\subset X\backslash\{x_0\}$ is a rational curve
of degree $e$.

\begin{lem}\label{lem rational curve in smooth locus}
Let $X\subset\PP^{n+1}$ be a cubic hypersurface with an isolated
ordinary double point $x_0\in X$. Then there exists a free rational
curve $C\subset X\backslash\{x_0\}$ of degree $e\geq 1$.
\end{lem}
\begin{proof}
We only need to show that there is a rational curve of degree $e$
through a general point of $X\backslash\{x_0\}$ and contained in
$X\backslash\{x_0\}$. We first do this for $e=1$. In this case, we
see from the above discussion that we only to show the following

\textsl{Claim}: Let $y\in\PP^n$ be a general point, then there is a
line $l'$ on $\PP^n$ that meets $Z$ in two points.

To prove the claim, we consider the projection,
$pr_y:\PP^n\backslash\{y\}\rightarrow\PP^{n-1}$, from the point $y$.
If there is no such line $l'$, then $pr_y|_Z:Z\hookrightarrow
\PP^{n-1}$ is a closed immersion. But this would imply that $Z$ is a
degree $6$ hypersurface in $\PP^{n-1}$ and hence $-K_Z=(n-6)H$ where
$H$ is the class of a hyperplane section of $Z$. Since
$Z\subset\PP^{n}$ is a $(2,3)$-complete intersection, we also know
that $-K_Z=(n+1-2-3)H=(n-4)H$, which is a contradiction.

For case of $e>1$, we can take a chain of $e$ lines on
$X\backslash\{x_0\}$ and smooth to a degree $e$ rational curve; see
\cite[\S II.7]{kollar}. Such a rational curve passes through a
general point since a line does.
\end{proof}

\begin{lem}
Under the situation of Proposition \ref{Lefschetz pencil}, there
exists a contractible analytic open neighborhood $D$ of $0\in B$
containing $\{t_1,\ldots,t_N\}$ such that

(i) There is an analytic family of curves $\{C_t\subset
X_t^{sm}:t\in D\}$ with $C_0=C$, where $X_t^{sm}$ is the smooth locus of $X_t$.

(ii) $\mathcal{S}=\cup_{t\in D}S_{C_t}\subset\mathcal{F}$ is a complex
manifold and $\rho:\mathcal{S}\rightarrow B$ is smooth away from the $t_i$'s;

(iii) As sub-manifolds of $\mathcal{F}$, $\mathcal{S}$ meets
each $\Gamma_i$ transversally at finitely many points
$z_{i1},\ldots,z_{il}$.
\end{lem}

\begin{proof}
Let $\mathcal{H}/B$ be the component of relative Hilbert scheme of
rational curves on $X_t$'s that contains $C$ as an element. This is
well-defined since $C$ defines a smooth point in the relative
Hilbert scheme. We claim that, after shrinking $D$, there exists a
general analytic section $s:t\mapsto C_t\in\mathcal{H}$ satisfies
$C_t\subset X_t^{sm}$. To prove this we first note that by Lemma
\ref{lem rational curve in smooth locus}, there is a free rational
curve $C_i\in X_{t_i}^{sm}$ of the same degree $e$. Since $C_i$ is
free, it deforms to a curve $C_t$ in nearby fibers. By the result of
\cite[Theorem 1.1]{cs}, the space of degree $e$ rational curves on a smooth cubic
hypersurface is irreducible. This implies that $C_t$ is a point of
$\mathcal{H}$ and hence so is $C_i$. So we can choose a family
$s:t\mapsto C_t$ whose specializations does not meet the points
$x_i\in X_{t_i}$. This proves (i). We also get (iii) easily by
choosing $s$ general enough. The smoothness of $\mathcal{S}$ follows
from a deformation argument. We only need to show that $\mathcal{S}$
is smooth at the points $z=z_{ij}$. Let $L$ be the line corresponds
to $z$ and $C_i=C_{t_i}$. Let
$v\in\mathrm{H}^0(C_i,\mathscr{N}_{C_i/\PP^{n+1}})$ be the section
corresponding to $\frac{\partial}{\partial t}$ at the point $C_i$.
Let $y=C_i\cap L$. Then
$v(y)\in\mathscr{N}_{C_i/X,y}=\frac{\mathcal{T}_{X,y}}{\mathcal{T}_{C_i,y}}$
determines a 2-dimensional subspace $V_y\subset \mathcal{T}_{X,y}$.
Then $\mathcal{T}_{\mathcal{S},z}\subset
\mathcal{T}_{\mathcal{F},z}$ is naturally given by all sections
$v'\in \mathcal{T}_{\mathcal{F},z}
\subset\mathrm{H}^0(L,\mathscr{N}_{L/\PP^{n+1}})$ with
$v'(y)\in\bar{V}_y$, where $\bar{V}_y$ is the image of $V_y$ in
$\frac{\mathcal{T}_{X,y}}{\mathcal{T}_{L,y}}=\mathscr{N}_{L/\PP^{n+1},y}$.
When $v(y)$ is general, the above condition gives a codimension
$n-2$ subspace of $\mathcal{T}_{\mathcal{F},z}$, i.e.
$\dim\mathcal{S}=\dim \mathcal{T}_{\mathcal{S},z}$. This proves
(ii).
\end{proof}

\begin{rmk}
From the above proof, we see that the family $t\mapsto C_t$ can be
made algebraic on some finite cover of $B=\PP^1$. The points
$z_{ij}$ are exactly the critical points of $\rho$ and all of them
are non-degenerate.
\end{rmk}

Now we fix a small $\epsilon$ and let $B_{i}$ be the closed ball of
radius $\epsilon$ with center $t_i$. If $\epsilon$ is small enough,
we have $B_i\subset D$. Pick a path $l_i$ connecting $0$ and
$t_i+\epsilon$ such that $\cup_il_i$ is star-shaped and
contractible. Let $D_i\subset\PP^{n+1}$ be a small open ball
centered at the double point $x_i\in X_{t_i}$. Let
$\tilde{D}_i\subset \mathcal{F}$ be the set of lines
$L\in\mathcal{F}$ such that $L$ meets $D_i$. Thus $\tilde{D}_i$ is a
small open neighborhood of $\Gamma_i$. By construction,
$$
\mathcal{S}\cap \tilde{D}_i=U_1\cup\cdots\cup U_l
$$
where $U_j$ are disjoint open balls in $\mathcal{S}$. Let
$p_j:P_j\rightarrow U_j$ be the family of lines parameterized by
$U_j$. Hence we get the following commutative diagram
$$
\xymatrix{
 P_j\ar[r]^{q_j}\ar[d]_{p_j} &Y\ar[d]^{\pi}\\
 U_j\ar[r]^{\rho_j} &B
}
$$
where $Y$ is the blow up of $\PP^{n+1}$ along the base locus of the
Lefschetz pencil and $\rho_j=\rho|_{U_j}$. We fix a general analytic
section $s_j:U_j\rightarrow P_j$.

\begin{lem}\label{local coordinates}(\cite[Lemma 7]{shimada})
There exists local analytic coordinates $u_0,u_1,\ldots,u_n$ of
$\PP^{n+1}$ at the point $x_j$ such that\\
\indent (i) The image of $P_j$ in $\PP^{n+1}$ is locally given by
$u_0+\sqrt{-1}u_1=0$;\\
\indent (ii) The image $q_j\circ s_j(U_j)$ is given by
$u_0=u_1=0$;\\
\indent (iii) $\pi=t_j+u_0^2+u_1^2+\cdots+u_n^2$.
\end{lem}

\begin{proof}
We first show that $q_j$ has injective tangent map. For any tangent
vector $v\in \mathcal{T}_{U_j,z_j}$, it corresponds to a global
section of $\mathscr{N}_{L_j/\PP^{n+1}}$, where $L_j$ is the line
corresponding to $z_j$. By the choice of $C_t$, we know that
$C_{t_i}$ meets the divisor swept out by lines through $x_i$
transversally. Then $v$ does not vanish at the point $x_i\in L_j$.
This means that the point $x_i\in L_j$ actually moves when we move
$L_j$ in $U_j$. Hence we get the injectivity of the tangent map of
$q_j$. Then we can start with local coordinates on $U_j$ and extend
to $P_j$ and then to $\PP^{n+1}$. See the proof of \cite[Lemma
7]{shimada} for more details.
\end{proof}

Since we are doing local computations, by abuse of notation, we
regard $P_j$ as a submanifold of $\PP^{n+1}$. Using the local
coordinates obtained in Lemma \ref{local coordinates}, it is
standard that a vanishing cycle associated to $x_i$ is given by
$$
\Sigma_{ij}=\{u_0^2+\cdots+u_n^2=\epsilon, \quad u_k\in\R\}
$$
See \cite{lamotke} for more details. Note that the index $j$ is used
to remember that the local coordinates are chosen with respect to
$P_j$ as in Lemma \ref{local coordinates}. If we fix an orientation
of $\Sigma_{ij}$ and this gives an element
$[\Sigma_{ij}]\in\mathrm{H}_n(X_{t_i+\epsilon},\Z)$. Also, by Lemma
\ref{local coordinates}, we know that $\sigma_{ij}=(q_j\circ
s_j)^{-1}\Sigma_{ij}$, with the induced orientation, gives a
vanishing cycle for the critical point $z_{ij}\in \mathcal{S}$. Let
$$
[\sigma_{ij}]\in \mathrm{H}_{n-2}(\mathcal{S}_{t_i+\epsilon},\Z)
$$
be the corresponding homology class. We choose the orientation of
$\Sigma_{ij}$ such that $[\Sigma_{ij}]=[\Sigma_i]$ is a fixed class.
Then $\sigma_{ij}$ has an induced orientation.

\begin{prop}(\cite[Proposition 2]{shimada}) For any
$\fa\in\mathrm{H}_{n-2}(\mathcal{S}_{t_i+\epsilon},\Z)$, we have
$$
\fa\cdot([\sigma_{i1}]+\cdots+[\sigma_{il}])=\Psi(\fa)\cdot[\Sigma_i]
$$
where $\Psi=\Psi_{C_{t_i+\epsilon}}$.
\end{prop}

\begin{proof}
The proof goes the same as that of \cite[Proposition 2]{shimada}; we
sketch it here for completeness. We still use $\fa$ to denote a
topological cycle that represents the class $\fa$. We use the
notations above and set $\fa_j=\fa\cap U_j$. If
$L\in\fa-\cup_j\fa_j$, then $L\cap\Sigma_i=0$. We may assume that
$\fa_j$ meets $\sigma_{ij}$ transversally at $\mu$ points
$a_1,\ldots,a_\mu\in\sigma_{ij}$. Then the intersection of
$\cup_{L\in\fa_j}L$ and $\Sigma_{ij}$ are exactly
$s_j(a_1),\ldots,s_j(a_\mu)$. And by construction, this intersection
is transversal. By chasing the orientations, we see that the
contributions to the two sides of the identity are equal.
\end{proof}

\begin{rmk}
Since $\cup_il_i$ is contractible,
$\mathrm{H}_n(X_{t_i+\epsilon},\Z)$ is naturally identified with
$H_n(X,\Z)$ via $(l_i)_*$. We can similarly identify
$\mathcal{S}_{t_i+\epsilon}$ and $S_C=\mathcal{S}_0$. Hence we can
view $[\sigma_{ij}]$ as elements in $V_{n-2}(S_C,\Z)$ and
$[\Sigma_i]$ as elements in $V_n(X,\Z)$. Under this identification,
the above proposition still holds true.
\end{rmk}

Now we are ready to prove the main result of this section. The proof
follows that of \cite[Proposition 4]{shimada}.
\begin{proof} (of Theorem \ref{surjectivity on primitive}). By
Lefschetz theory, $V_{n}(X,\Z)$ is generated by vanishing cycles,
see \cite{lamotke}. Since $\Psi=\Psi_C$ commutes with the
specialization map, we know that $\Psi([\sigma_{i1}])=\lambda
[\Sigma_i]$ for some $\lambda\in\Z$. Now we see that
$$
[\sigma_{i1}]\cdot([\sigma_{i1}]+\cdots+[\sigma_{il}])=[\sigma_{i1}]\cdot[\sigma_{i1}]=\pm
2
$$
See \cite{lamotke}, p.40. By the above proposition, we get
$$
\Psi([\sigma_{i1}])\cdot[\Sigma_i]=\lambda[\Sigma_i]\cdot[\Sigma_i]=\pm
2
$$
Comparing this with $[\Sigma_i]\cdot[\Sigma_i]=\pm 2$, we get
$\lambda=\pm 1$. This shows that $[\Sigma_i]$ is in the image of
$\Psi$. This proves the theorem since $[\Sigma_i]$ generates
$V_n(X,\Z)$.
\end{proof}

\begin{rmk}\label{rmk n equals 4}
To complete the picture, we still need to prove the surjectivity of
$$
\Psi_C:\mathrm{H}^{n-2}(S_C,\Z)^\circ\to\mathrm{H}^n(X,\Z)_{prim}
$$
for $n=4$. Note that under Poincar\'e duality,
$\mathrm{H}^{2}(S_C,\Z)^\circ\subset\mathrm{H}^{2}(S_C,\Z)$ consists
of all classes in $V_2(S_C,\Z)$ which is orthogonal to $g'$. Hence
we only need to show that $\sigma_{ij}\cdot g'=0$. But this is easy
to see from construction. In fact, we pick a general point $x\in
C_{t_i+\varepsilon}$, then $g'$ is represented by $F_x$ (all lines
through $x$). Then by construction, all lines through $x$ avoids
$\sigma_{ij}$ (this is essentially due to the fact that the surface
swept out by lines in through $x$ is $2h^2$; see Lemma \ref{multiple
of h lemma}). This means $\sigma_{ij}\cdot g'=0$.
\end{rmk}

\end{document}